\documentclass[11pt]{article}

\usepackage{amsmath} 
\usepackage{amssymb, amsthm}  
\usepackage{amsfonts,dsfont, bm, eufrak} 
\usepackage{color}

\usepackage{tikz} 
\usepackage{circuitikz}
\usetikzlibrary{arrows,shapes,backgrounds,decorations.text}

\usetikzlibrary{backgrounds}
\pgfdeclarelayer{myback}
\pgfsetlayers{background,myback,main} 

\newcommand*{\tikzS}{3}%
\newcommand*{\tikzSS}{6}%

\newcommand*{\tikzSSS}{4.2}%

\usepackage{datetime}
\usepackage{enumitem}

\usepackage{graphicx}
\graphicspath{}
\newlength{\figtriml}
\newlength{\figtrimb}
\newlength{\figtrimr}
\newlength{\figtrimt}
\newlength{\figwidth}
\newcommand{\white}{\color{white}}

\newcommand{\eabullet}{\hfill$\bullet$}

\newcommand{\bb}{\mathbf{b}}

\newcommand{\bias}{\mathfrak{f}}
\newcommand{\xbias}{x_\bias}
\newcommand{\ybias}{y_\bias}

\DeclareMathOperator*{\diag}{Diag}

\newcommand{\vecg}{\succ}

\newcommand{\vecl}{\prec}
\newcommand{\vecleq}{\preccurlyeq}

\newcommand{\1}{\mathbf{1}}
\newcommand{\0}{\mathbf{0}}

\newcommand{\IDE}{\mathbb{I}}

\newcommand{\R}{\mathbb{R}}
\newcommand{\N}{\mathbb{N}}
\newcommand{\Rp}{\mathbb{R}_+}

\newcommand{\NN}{\mathcal{N}}
\newcommand{\GG}{\mathcal{G}}
\newcommand{\MM}{\mathcal{M}}

\newcommand{\CC}{\mathcal{C}}
\newcommand{\DD}{\mathcal{D}}

\newcommand{\rr}{\mathbf{r}}
\renewcommand{\ss}{\mathbf{s}}

\newcommand{\yy}{\mathbf{y}}

\newcommand{\xx}{\mathbf{x}}

\newcommand{\Qij}{Q_{ij}}

\newcommand{\qi}{q_i}
\newcommand{\qq}{\mathbf{q}}

\newcommand{\aalpha}{\boldsymbol{\alpha}}
\newcommand{\bbeta}{\boldsymbol{\beta}}
\newcommand{\eeta}{\boldsymbol{\eta}}
\newcommand{\oomega}{\boldsymbol{\omega}}

\newcommand{\Rl}{{R^\ell}}

\newcommand{\yyRl}{\yy_{\Rl}}
\newcommand{\LRl}{L_{\Rl,\Rl}}

\newcommand{\vv}{\mathbf{v}}

\newcommand{\ones}{\mathbf{1}}
\newcommand{\zeros}{\mathbf{0}}
\newcommand{\hh}{\mathbf{h}}

\newcommand{\cc}{\mathbf{c}}
\newcommand{\ww}{\mathbf{w}}

\newcommand{\Wil}{W^{i\to \ell}}
\newcommand{\Hil}{H^{i\to \ell}}
\newcommand{\Wij}{W^{i\to j}}
\newcommand{\Hij}{H^{i\to j}}
\newcommand{\Wki}{W^{k\to i}}
\newcommand{\Hki}{H^{k\to i}}

\newtheorem{theorem}{Theorem}

\newtheorem{lemma}[theorem]{Lemma}

\newtheorem{remark}{Remark}
\newtheorem{assumption}{Assumption}
 \title{The harmonic influence in social networks and its distributed computation by message passing}
 
\author{Wilbert Samuel Rossi and Paolo Frasca
\thanks{W. S. Rossi  is with the Department of Applied Mathematics,
        University of Twente, 7500 AE Enschede, The Netherlands
        {\tt\small w.s.rossi@utwente.nl}; P. Frasca is with Univ. Grenoble Alpes, CNRS, Inria, GIPSA-lab, F-38000 Grenoble, France
        {\tt\small paolo.frasca@gipsa-lab.fr}}%
}

\begin{document}
\maketitle
\thispagestyle{empty}
\pagestyle{empty}

\begin{abstract} In this paper we elaborate upon a measure of node influence in social networks, which was recently proposed by Vassio et al., IEEE Trans. Control Netw. Syst., 2014. This measure quantifies the ability of the node to sway the average opinion of the network.  
Following the approach by Vassio et al., we describe and study a distributed message passing algorithm that aims to compute the nodes' influence. The algorithm is inspired by an analogy between potentials in electrical networks and opinions in social networks. If the graph is a tree, then the algorithm computes the nodes' influence in a number of steps equal to the diameter of the graph. On general graphs, the algorithm converges asymptotically to a meaningful approximation of the nodes' influence. In this paper we detail the proof of convergence, which greatly extends previous results in the literature, and we provide simulations that illustrate the usefulness of the returned approximation.
\end{abstract}



\section{Introduction} 

In the study of networks and dynamical processes therein, one important issue is  the identification of the most influential nodes, i.e. those with the higher ability to drive the others towards a desired state. 
The issue depends on the process and the control objective: consequently, it has been addressed in several contexts, from the seminal paper~\cite{kempe2003maximizing} on maximizing the spreading of influence, to several leader selection problems recently considered, such as~\cite{easley2010networks,Estrada20103648,clark2014minimizing,Yildiz:2013,Vassio:2014:journal,linfarjovTAC14leaderselection,7277027}. 

In this work, we assume that a leader has to compete against an external field of influence in order to win the opinions of the other individuals in a social network. 
Following a consolidated research line~\cite{NEF-ECJ:99,como:opinion,parsegov2015new}, we shall postulate that the opinions of the nodes follow a linear dynamics with fixed confidence weights.

More precisely, the leader node has a fixed opinion, whereas the remaining ``regular'' agents update their opinions, in synchronous rounds, to a weighted average between the opinions of their neighbors and the external field. The opinions converge asymptotically to values that depend on the interaction strength, on the opinion field, on the leader location and opinion, but not on the initial opinion of the regular agents.
Assuming for simplicity that the leader has opinion one and the external field has opinion zero, we define the \textit{influence} of the leader as the sum of the asymptotic opinions reached by the agents in the social network. The influence of a node is the influence obtained if that node was the leader. As we will discuss in details later, this definition is very close to the definition of \textit{Harmonic Influence Centrality} introduced in~\cite{Vassio:2014:journal} and implicitly in~\cite{Yildiz:2013}.

If the interaction strengths satisfy a ``reciprocity'' assumption (essentially, the reversibility of the update matrix), then the asymptotic opinions of the non-leader agents can be computed as the \textit{harmonic} electrical potential in an electrical network obtained from the social graph.
Consequently, the harmonic influences of the nodes can be computed exactly by solving an array of $N$ linear systems defined by the Laplacian of the graph, ``grounded'' in each of the $N$ nodes~\cite{6859421:sundaram}. This straightforward approach proposed already in~\cite{Yildiz:2013} has some drawbacks. 
Firstly, global knowledge of the graph and update matrix is required by most solution methods, with the exception of some distributed (i.e. non-global) methods like~\cite{SSWBD:2008} and~\cite{7063919}.
Secondly, solving $N$ systems is computationally expensive, even if one can resort to state-of-the-art algorithms that are tailored to Laplacian systems: these methods can solve each system in a time proportional to the number of edges but are not distributed~\cite{nkV:2012}.
Moreover, since the $N$ systems are obtained by grounding the same original Laplacian, solving them separately seems to produce wasteful redundancies in the computations.

Along the lines of~\cite{Vassio:2014:journal}, in this paper we take an approach that overcomes all these issues by studying a  \textit{Message Passing Algorithm} (MPA) able to concurrently compute the influence of all nodes. 
Our algorithm is {\em distributed}, that is, does not require any global knowledge of the graph or of the parameters of the opinion dynamics: moreover, it computes the harmonic influence of all nodes at the same time.

If the graph is a tree, then the algorithm computes the nodes' influence in a number of steps equal to the diameter of the graph. On general graphs, the algorithm converges asymptotically. The main result of this paper is indeed this proof of convergence, which subsumes the available results for regular graphs~\cite{Vassio:2014:journal} and for unicyclic graphs~\cite{RF:2016:ecc}. It must be stressed that in general the algorithm, even though it converges, does not converge to the exact values of the influence. In order to explore the magnitude of this error, we conclude the paper by simulating the MPA on synthetic graphs. We observe that when the number of cycles increases, the algorithm becomes slower and less accurate, but nevertheless provides a useful approximation of the harmonic influence.

\subsubsection*{Relations with the message passing literature} 
Our paper contributes to the literature on message passing algorithms, by providing an interesting example of algorithm that converges on {\em any} graph~\cite{mezard2009information}. On the contrary, proofs of convergence of message passing algorithms are often limited to tree graphs or to locally-tree-like graphs.

In this field, a closely related paper is~\cite{SSWBD:2008}, which reformulates the problem of solving a linear systems $A\xx = \bb$, where the matrix $A$ is full rank and symmetric, in a probabilistic inference problem. Then, it develops a gaussian belief propagation method, that involves two kinds of messages. The authors prove that if $A$ is strictly diagonally dominant, or if the matrix $|\IDE - A|$ (where $\IDE$ is the identity matrix) has spectral radius strictly smaller than one, then the algorithm converges to the exact solution. On trees, the algorithm coincides with the direct gaussian elimination method.

Our work also shares some ideas with~\cite{MVR:2006}, which proposes a \textit{consensus propagation} protocol based on two kinds of messages to solve the consensus problem: one contains a partial estimate of the consensus value and the other contains the number of nodes involved in such partial estimate. A suitable attenuation parameter makes the protocol~\cite{MVR:2006} convergent on general graphs.

\subsubsection*{Paper Structure}
Section~\ref{sect:opi_model} formulates the underlying opinion dynamic model with our standing the assumptions. Section~\ref{sec:elec} describes the electrical interpretation and discusses the equivalence with the problem formulated by~\cite{Yildiz:2013,Vassio:2014:journal,Vassio:2014:ECC}.
Section~\ref{sec:mpa} describes our Message Passing Algorithm for the concurrent and distributed computation of the harmonic influence, whereas the technical proof of convergence is given in Section~\ref{sect:convergence}. 
After some simulations presented in Section~\ref{sec:simulations}, Section~\ref{sec:conclu} concludes the paper.

\subsubsection*{Notation} 
The set of real and non-negative real numbers are denoted by $\R$ and $\Rp$, respectively.
Vectors are denoted with boldface letters and matrices with capital letters. 
The all-zero and all-one vectors are denoted by $\0$ and $\1$, respectively.
The symbol $\IDE$ denotes any identity matrix with appropriate dimension.
The symbol $\vecleq$ denotes entry-wise $\leq$ for vectors and matrices.  
The symbol $\vecl$ is used if the entry-wise inequality is strict for at least one entry. 
Given a matrix $Q$, $Q^\top$ denotes its transpose, $Q^{-1}$ its inverse,  $Q^\dagger$ its pseudo-inverse and $\rho(Q)$ its spectral radius, i.e. the maximum absolute value of the eigenvalues of $Q$. If $\rho(Q) <1$, $Q$ is termed ``Schur stable''. 
Given a vector $\vv$, $\diag(\vv)$ is the square diagonal matrix with the entries of $\vv$ on the main diagonal.
The cardinality of the set $S$ is denoted by $|S|$. The symbol $\subset$ is used for strict subsets; $\subseteq$ for generic subsets.
Given the matrix $Q \in \R^{S \times S}$ and the subsets $T,T' \subseteq S$, $Q_{T,T'}$ denotes the sub-matrix of $Q$ containing the rows and columns corresponding to $T$ and $T'$, respectively.
A non-negative matrix  $Q \in \Rp^{S \times S}$ is said to be stochastic, sub-stochastic and strictly sub-stochastic if $Q \1 = \1$,         $Q \1 \vecleq \1$ and $Q \1 \vecl \1$, respectively.

Let $\GG = (V,E)$ be a graph where $V$ is the set of vertices and $E$ is the set of edges, that are unordered pairs of vertices. We will use the terms \textit{node}, \textit{vertex} and \textit{agent} interchangably. 
The set $N_v = \{w \in V: \{v,w\}\in E \}$ contains the neighbors of $v$ in $\GG$; the degree of $v$ is $d_v = |N_v|$. A \textit{leaf} is a node of degree one.
We say that $\GG' = (V',E')$ is a subgraph of $\GG = (V,E)$ if $V' \subset V$ and $E' \subset E$. If $\GG'$ contains all edges of $\GG$ that join two vertices in $V'$, then $\GG'$ is said to be the subgraph \textit{induced by} $V'$ and is denoted by $\GG[V']$. 
A \textit{path} joining the nodes $v$ and $w$ is an ordered list of nodes $(u_0, u_1, \ldots, u_l)$ such that: 
\begin{enumerate}
\item[(i)] $u_0 = v$ and $u_l = w$;
\item[(ii)] $\{u_{i-1},u_{i}\} \in E$ for every $i \in \{1,\ldots, l\}$. 
\end{enumerate}
The \textit{length} of the path is $l$ and counts the number of edges involved. 
We term \textit{simple} a path where the edges are all distinct. 
We term \textit{circuit} a simple path of length $l \geq 3$ where the first and last node coincide. 
The graph $\GG$ is \textit{connected} if for any pair of nodes $v$, $w \in V$ there exists a simple path joining them.

\section{Opinion Dynamics Model}\label{sect:opi_model}
In this section we formulate the opinion dynamic model and consequently define the influence of the nodes. 

Consider a simple, undirected and connected graph $\GG = (I,E)$ with node set $I$ of cardinality $n$ and edge set $E$. 
Each node is an agent, endowed with a scalar opinion $x_i(t) $ of initial value $x_i(0) \in \R$, to be updated at discrete time steps $t \in \N$.  
All opinions are stacked in the vector $\xx(t) \in \R^I$.

The agent $\ell \in I$ is a stubborn \textit{leader} that never changes opinion:
$$x_\ell(t) = x_\ell(0) \qquad \forall t >0 .$$ 
The remaining \textit{regular} agents in $R^\ell := I \setminus \{\ell\} $ update their opinions to a convex combination of the opinions of their neighbors and the constant external opinion field $\xbias \in \R$.
Consider a sub-stochastic matrix $Q \in \Rp^{I\times I}$ and a non-negative vector $\qq \in \Rp^I$ such that:
\begin{align}
\label{eq:Q-q-struct}
	\left\{\begin{array}{l} 
		\Qij =0\Leftrightarrow  \{i,j\}\notin E \\
		\sum_j \Qij + q_i = 1  
	\end{array}\right. 
	\qquad \forall i \in R   \,.
\end{align}
The weights $\Qij$ and $\qi$ represent how much the agent $i$ trusts the agent $j$ and the opinion field, respectively. 
Each regular agent $i \in \Rl$ updates its opinion following:
$$x_i(t+1) = {\textstyle \sum_{j \in I} } \, \Qij \, x_j(t) + q_i \,x_\bias \qquad \forall t \geq 0 \,.$$
In compact form the  regular agents' update rule is:
$$	\xx_{R^\ell}(t+1) = Q_{R^\ell, R^\ell} \xx_{R^\ell}(t) + Q_{R^\ell, \{\ell\}} x_\ell(0) +\qq_{R^\ell} \xbias \quad  \forall t\geq 0\,. $$ 
We then make the following assumption to avoid trivial cases.
\setcounter{assumption}{-1}
\begin{assumption}\label{ass:zero}
The vector $\qq$ is not identically zero while, for any $\ell$, the matrix $Q_{R^\ell, R^\ell}$ is strictly sub-stochastic.
\eabullet\end{assumption}
Thanks to the connectivity of $\GG$, this assumption implies that $Q_{R^\ell, R^\ell}$ is a Schur stable matrix~\cite[Lemma~5]{PF-CR-RT-HI:13c}. 
Given the leader's opinion $x_\ell(0)$ and the opinion field $\xbias$, 
the regular agents' opinions tend to the limit $\xx_{R^\ell}(\infty)$, unique solution of:
\begin{align} 
	\xx_{R^\ell}(\infty) & = Q_{R^\ell, R^\ell} \xx_{R^\ell}(\infty) +  Q_{R^\ell, \{\ell\}} x_\ell(0) +\qq_{R^\ell} \xbias,  \label{eq:asy-R-eq}
\end{align}
and convex combinations of $x_\ell(0)$ and $\xbias$.

Since we want to take advantage of an electrical analogy, we need a ``reciprocity'' assumption (similar to reversibility) on the weights of the matrix $Q$ and vector $\qq$. 
For notational convenience, we define the extended set $I_\bias := I \cup \{\bias\}$. 
\begin{assumption}\label{ass:Qspeciale}
There exist a symmetric matrix $C \in \Rp^{I_\bias \times I_\bias}$ such that for every pair $i,j \in I $\,: 
$$ Q_{ij} = \frac{C_{ij} }{ \sum_{k \in I_\bias} C_{ik}} \,, \qquad q_i = \frac{C_{i\bias} }{ \sum_{k \in I_\bias} C_{ik}}\,. $$ 
\eabullet\end{assumption}
The non-negative symmetric matrix $C$ is called \textit{conductance matrix} and will be used in the following. 
For $i,j$ in $\Rl$, the values $C_{ij} = C_{ji}$ represent a measure of the strength of the reciprocal relation between $i$ and $j$. 
The values of $C_{\ell j}$ and $C_{\bias j}$ do not play any role, although they are fixed by the symmetry of $C$.  
Note that $C_{ij} > 0$ if $\{i,j\} \in E$ while $C_{i\bias }>0$ if $q_i >0$.

By the linearity of the dynamics, we can without loss of generality choose the following initial conditions.
\begin{assumption}\label{ass:null-Z-bias} The opinion of the leader is $x_\ell(0) = 1$; the value of the opinion field is $\xbias = 0$.
\eabullet\end{assumption}
Under these assumptions, we can define the \textit{harmonic influence} of $\ell$ as the sum of the asymptotic opinions of all agents, leader included:  
\begin{align*}
	H(\ell) := 1 + \1^\top \xx_\Rl(\infty)\,.
\end{align*}

Assumptions~\ref{ass:zero},~\ref{ass:Qspeciale}, and~\ref{ass:null-Z-bias} will hold in the rest of the paper.

\section{The Electrical Interpretation}\label{sec:elec}
 
In light of Assumptions~\ref{ass:zero} and~\ref{ass:Qspeciale}, the opinion dynamic model  
is intimately related to the linear circuit theory: the asymptotic opinions $\xx_\Rl(\infty)$ can be interpreted as electrical potentials in a suitably defined \textit{electrical network}.
An {electrical network} is a pair $(\NN, C)$, where $\NN = (J,F)$ is a simple connected graph and $C\in \Rp^{J\times J}$ is a non-negative, symmetric conductance matrix. The conductance matrix is ``adapted'' to the graph (i.e. $C_{ij} = C_{ji} = 0 \Leftrightarrow \{i,j\} \notin F$) 
and each edge $\{i,j\} \in F$ has electrical conductance $C_{ij}$. 

We describe the electrical network $(\NN,C)$ corresponding to the our opinion dynamic model. 
Consider the extended node set $I_\bias := I\cup\{\bias\}$ and the new graph $\NN = (I_\bias, E\cup E_\bias)$, obtained adding to $\GG$ the \textit{reference} node~$\bias$ and the edges in $E_\bias = \left\{ \{\bias,i\} : i\in I,\, q_i >0 \right\}$.
Following Assumptions~\ref{ass:zero}, the graph $\NN$ is connected by construction. 
The pair $(\NN, C)$, where $C$ is the conductance matrix of Assumption~\ref{ass:Qspeciale}, is the electrical network corresponding to our opinion dynamic model.

The asymptotic opinions of the regular agents, solution of \eqref{eq:asy-R-eq}, coincide with the electrical potentials that solve the circuit equations of the network.
\begin{lemma}[Potentials and opinions]\label{lem:elec-1}
Consider the opinion dynamic model on the graph $\GG$ with leader $\ell$ of initial opinion $x_\ell(0)$ and opinion field $\xbias$, and let Assumptions~\ref{ass:zero} and~\ref{ass:Qspeciale} hold. 
Consider the electrical network $(\NN,C)$ described above and the vector $\yy \in \R^{I_\bias}$ that contains the nodes' electrical potentials.
If the potentials of the node $\ell$ is held at $y_\ell = x_\ell(0)$ and the potential of the node $\bias$ is held at $\ybias = \xbias$, then $\yy_{R^\ell} = \xx_{R^\ell}(\infty) $.
\end{lemma}
\begin{proof}
Given Assumption~\ref{ass:zero} and Assumption~\ref{ass:Qspeciale}, the electrical network $(\NN,C)$ is connected and contains two nodes that have fixed potential, namely $\ell$ and $\bias$.
The potential of the remaining nodes (in $R^\ell =  I_\bias \setminus \{\ell, \bias\}$) is uniquely determined by Kirchhoff's current law and Ohm's law. 
The system of $n-1$ independent \textit{node} equations is:
\begin{align} \label{eq:sys-node-eq}
	\forall i \in R^\ell \quad \sum_{j\in I_\bias} C_{ij} (y_i - y_j) = 0\,. 
\end{align}
Dividing each side for $\sum_{k \in I_\bias} C_{ik}$, we recognize the elements of $Q$ and $\qq$: 
\begin{align*}
	& \forall i \in R^\ell  \quad  \sum_{j\in I} \Qij (y_i - y_j) + \qi (y_i - \ybias) = 0\,, 
\intertext{and since $\qi + \sum_{j\in I}\Qij = 1$ for every $i \in R^\ell$, we get: }
	&   \forall i \in R^\ell  \quad y_i =  \sum_{j\in I \setminus \{\ell\}} \Qij y_j  + Q_{i\ell} y_\ell + \qi \ybias\,, 
\end{align*}
that is: $$\yy_{\Rl} = Q_{\Rl,\Rl} \yy_{\Rl} + Q_{\Rl,\{\ell\}}  y_\ell + \qq_{\Rl} \ybias\,.$$
If $y_\ell = x_\ell(0)$ and $\ybias = \xbias$, the vector $\yy_{R^\ell}$ coincides with $\xx_{R^\ell}(\infty)$ because it satisfies the set of equations \eqref{eq:asy-R-eq}.
\end{proof}

We introduce the \textit{Laplacian matrix} $L(C) \in \R^{I_\bias \times I_\bias}$ associated to the conductance matrix $C$:   $$L(C) = \diag( C \1) - C\,.$$
Assumption~\ref{ass:null-Z-bias} fixes the values of the leader opinion and the opinion field, and hence the potential of the  leader node $\ell$ and opinion field node~$\bias$.
The potential of the nodes in $\Rl$ can be computed by solving a linear, \textit{Laplacian} system with boundary condition, or by inverting the sub-matrix $L(C)_{\Rl,\Rl}$. 

\begin{lemma}[Potentials and Laplacians]\label{lem:elec-2}
Consider the electrical network $(\NN, C)$ described above, where the potentials of node $\ell$ and $\bias$ are fixed. 
Let Assumption~\ref{ass:zero},~\ref{ass:Qspeciale} and~\ref{ass:null-Z-bias} hold. 
The vector $\yy$ that contains the potentials of the electrical network is the solution of the Laplacial system with boundary conditions:
\begin{align} \label{eq:lap-sys} \left\{ \begin{array}{l} 
	\left( L(C) \, \yy \right)_{R^\ell}= \0 \\ 
	y_\ell = 1 \\
	y_\bias = 0 \,. \end{array}\right. \end{align}
The potentials of the nodes in $\Rl$ can also be computed as: 
$$	\yyRl = (L(C)_{\Rl,\Rl})^{-1} \, C_{\Rl,\{\ell\} } \, .$$
\end{lemma}
\begin{proof}
The $i^{\text{th}}$ equation  in  the system \eqref{eq:sys-node-eq}, where $i \in \Rl$, reads:  
$$ {\textstyle y_i \sum_{j\in I_\bias} C_{ij} - \sum_{j\in I_\bias} C_{ij} y_j = 0} \,,$$ and can be rewritten as:
$$   (( \diag(C\1 ) - C ) \,\yy )_i = 0 \,.  $$ 
Given the boundary conditions and the definition of $L(C)$ we recognize the Laplacian system above.
Then: 
$$L(C)_{\Rl,\Rl} \, \yy_\Rl +  L(C)_{\Rl,\{\ell\}} y_\ell +  L(C)_{\Rl,\{\bias\}} y_\bias = \0\,,$$
which, using the boundary conditions,  gives: 
$$L(C)_{\Rl,\Rl} \, \yy_\Rl = -  L(C)_{\Rl,\{\ell\}} = C_{\Rl,\{\ell\}} \,,$$ 
because $C_{ij} = -L_{ij}$ if $i\neq j$.
The result follows since $\NN$ is connected and $\emptyset \subset \Rl \subset I_\bias$ so the matrix $\LRl$ is positive definite~\cite{4177510} and hence invertible. 
\end{proof}
The potential of the nodes in  $\Rl$ is said to be the \textit{harmonic} extension of the  potential fixed at the leader and opinion field nodes to the regular nodes. 
Using the electrical interpretation -- Lemmas~\ref{lem:elec-1} and~\ref{lem:elec-2} --
the \textit{harmonic influence} of $\ell$ can be expressed as: 
\begin{align*}
	H(\ell) &= 1 + \1^\top \yyRl \nonumber \\
	&=1 + \1^\top \left(L(C)_{\Rl,\Rl}\right)^{-1} \, C_{\Rl,\{\ell\} } \,.
\end{align*}

\subsection{Equivalence with the formulation of~\cite{Vassio:2014:journal,Vassio:2014:ECC}} 
Using the electrical analogy, we now briefly recall the problem formulated in~\cite{Vassio:2014:journal} and show that it is equivalent to our problem. The advantage of our formulation is that it allows for an effective analysis of the Message Passing Algorithm presented in the next section.

The paper~\cite{Vassio:2014:journal} studies the \textit{harmonic influence} of the nodes in a graph where a few leaders with null opinion are already present. Basically, these leaders provide the ``opinion field'' that we introduced earlier in this paper, and vice-versa, our opinion field can be regarded as an additional leader with null opinion, to which some nodes are connected. The equivalence is straightforward from the point of view of the electrical interpretation.
For a detailed description of the equivalence, let $(\tilde \GG, \tilde C)$ be the electrical network in~\cite{Vassio:2014:journal} with node potentials $\tilde \yy$, where $\tilde \GG = (\tilde I, \tilde E)$ is a connected graph and  $\tilde C \in \Rp^{\tilde I\times \tilde I}$ is the conductance matrix. %
The nodes of a subset $S^0 \subset \tilde I$ have fixed, null potential, i.e. $\tilde \yy_{S^0} = \0$ and are assumed to be \textit{leaves}, without loss of generality.
The node $\ell \in \tilde I \setminus S^0$ has potential fixed to one, i.e. $\tilde y_\ell = 1$.
The potential of the remaining nodes, i.e. those in $\tilde R^\ell = \tilde I \setminus S^0 \setminus \ell $,  are determined by the Laplacian system with boundary conditions: 
\begin{align} \label{eq:lap-sys-vassio }\left\{ \begin{array}{l} 
	\left( L(\tilde C) \, \tilde \yy \right)_{\tilde R^\ell}= \0 \\ 
	\tilde y_\ell = 1 \\
	\tilde \yy_{S^0} = \0 \,. \end{array}\right.\end{align}

From the electrical network $(\tilde \GG, \tilde C)$ of~\cite{Vassio:2014:journal}, we obtain our  electrical network  $(\NN,C)$ by collapsing all the nodes in $S^0$ into a single node, recognized to be the opinion field node $\bias$. It is possible to collapse the node in $S^0$ (\textit{gluing} operation, see~\cite{Vassio:2014:journal})  because they all have the same potential.
Possible parallel edges created by the gluing operation shall be substituted by a single edge having conductance equal to the sum of the conductances of the parallel edges. The left and central scheme of Figure~\ref{fig:equiv-vassio} show this equivalence. 

The other direction of the equivalence is easier. Consider our graph $\GG = (I,E)$ and the graph $\NN = (I_\bias, E\cup E_\bias)$ of the electrical network $(\NN,C)$ associated. 
If $|E_\bias| = 1$, we recognize the opinion field node to be a leader with null opinion and set $(\tilde \GG, \tilde C) \equiv (\NN, C)$. 
If $|E_\bias| > 1$, to preserve the fact that the leaders with null opinion are leaves, we consider $|E_\bias|$ leader nodes with null opinion, and connect each of them to a different node previously connected to the opinion field node. 
The central and right scheme of Figure~\ref{fig:equiv-vassio} show this operation. The conductance matrix $\tilde C$ is obtained accordingly from $C$.

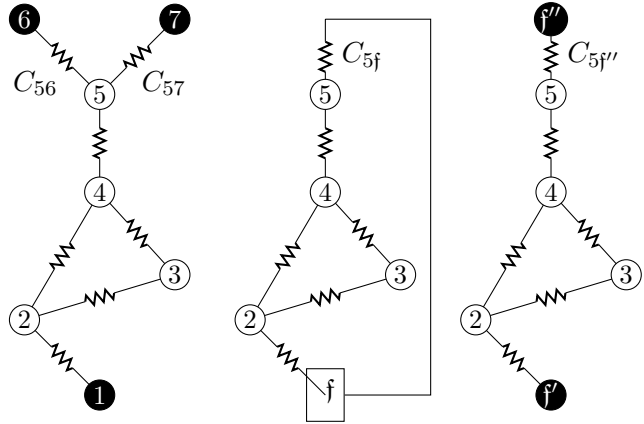
\begin{figure}
\centering
\ctikzset{ bipoles/length =.5 cm , label/align = straight}
\begin{tikzpicture}

	
	\tikzstyle{G node} =  [shape = circle, draw, fill = white, minimum size = 4mm, inner sep = 0mm];
	\tikzstyle{GZ node} = [shape = circle, draw, fill = black, minimum size = 4mm, inner sep = 0mm];


	\node[GZ node] (A1) at (1,0) {\white\small$1$};
	\node[G node] (A2) at (0,1) {\small$2$};
	\node[G node] (A3) at (2,1.6) {\small$3$};
	\node[G node] (A4) at (1,2.7) {\small$4$};
	\node[G node] (A5) at (1,4) {\small$5$};
	\node[GZ node] (A6) at (0,5) {\white\small$6$};
	\node[GZ node] (A7) at (2,5) {\white\small$7$};
	
\begin{pgfonlayer}{myback}
	\draw(A1) to [R] (A2); 
	\draw(A2) to [R] (A3); 
	\draw(A2) to [R] (A4); 
	\draw(A3) to [R] (A4); 
	\draw(A4) to [R] (A5); 
	\draw(A5) to [R, l^=$C_{56}$] (A6); 
	\draw(A5) to [R, l_=$C_{57}$] (A7); 
\end{pgfonlayer}
	
	\node[G node] (B2) at (0 + \tikzS ,1) {\small$2$};
	\node[G node] (B3) at (2 + \tikzS ,1.6) {\small$3$};
	\node[G node] (B4) at (1 + \tikzS ,2.7) {\small$4$};
	\node[G node] (B5) at (1 + \tikzS ,4) {\small$5$};
	\node[rectangle,draw] (B0) at (4,0) {~\small \raisebox{2mm}{$\bias$}}; 

\begin{pgfonlayer}{myback}
	\draw(B0) to [R] (B2); 
	\draw(B2) to [R] (B3); 
	\draw(B2) to [R] (B4); 
	\draw(B3) to [R] (B4); 
	\draw(B4) to [R] (B5); 
	\draw(B5) to [R, l_=$C_{5\bias}$] (1 + \tikzS ,5) -- (5.4,5) -- (5.4,0) --  (B0);
	
\end{pgfonlayer}	

	\node[GZ node] (C0) at (1 + \tikzSS ,0) {\white\small$\bias'$};
	\node[G node] (C2) at (0 + \tikzSS ,1) {\small$2$};
	\node[G node] (C3) at (2 + \tikzSS ,1.6) {\small$3$};
	\node[G node] (C4) at (1 + \tikzSS ,2.7) {\small$4$};
	\node[G node] (C5) at (1 + \tikzSS ,4) {\small$5$};
	\node[GZ node] (C000) at (1 + \tikzSS ,5) {\white\small$\bias''$};

\begin{pgfonlayer}{myback}
	\draw(C0) to [R] (C2); 
	\draw(C2) to [R] (C3); 
	\draw(C2) to [R] (C4); 
	\draw(C3) to [R] (C4); 
	\draw(C4) to [R] (C5); 
	\draw(C5) to [R , l_=$C_{5\bias''}$] (C000); 
\end{pgfonlayer}	

\end{tikzpicture}
\caption{\label{fig:equiv-vassio}  Equivalence with the formulation of~\cite{Vassio:2014:journal}. 
The leftmost and rightmost scheme follow the formulation of~\cite{Vassio:2014:journal}: the black nodes are the leaders with null opinion, in $S^0$. 
The central scheme is based on our formulation: the reference node (drawn with a rectangle) represents the opinion field, at null potential. The comparison between left and central scheme shows the mapping from~\cite{Vassio:2014:journal} to this paper; note that $C_{5\bias} = C_{56} + C_{57}$. The comparison between central and right scheme shows the mapping from this paper to~\cite{Vassio:2014:journal}; note that $ C_{5\bias'' } = C_{5\bias' } = C_{5\bias}$.} 
\end{figure}

\section{Distributed computation of the influence}\label{sec:mpa}

In this section we present a Message Passing Algorithm (MPA) to compute in a distributed way the influence of every node in $\GG = (I,E)$. 
In principle, the computation of the  \textit{harmonic influence} of the nodes in the set $I$ requires the solution of  $|I|$ linear systems of the form \eqref{eq:lap-sys}. 
This approach requires global knowledge of the graph and of the matrix $L(C)$: moreover, solving these $|I|$ systems independently does not exploit the apparent redundancies between them. Because of these drawbacks, in this paper we follow the approach by~\cite{Vassio:2014:journal} and we propose a distributed MPA for the distributed computation of the node's influence.
The intuition of the algorithm is provided by the electrical analogy and the messages exchanged in the algorithm have an exact interpretation if $\GG$ is a tree. In such a case, the algorithm is exact; else, we are going to prove that the algorithm converges to an approximation of the \textit{harmonic influence}.

\begin{figure}
\centering
\ctikzset{ bipoles/length =.5 cm , label/align = straight}
\begin{tikzpicture}

	
	\tikzstyle{G node} =  [shape = circle, draw, fill = white, minimum size = 4mm, inner sep = 0mm];
	\tikzstyle{GZ node} = [shape = circle, draw, fill = black, minimum size = 4mm, inner sep = 0mm];

	\coordinate (root) at (2,1);
	\coordinate (v1) at (2,2);
	\coordinate (v2r) at (3,3.5);
	\coordinate (v2l) at (1,3.5);
	\coordinate (v3rr) at (3.5,5);
	\coordinate (v3r) at (2.5,5);
	\coordinate (v3l) at (1.5,5);
	\coordinate (v3ll) at (0.5,5);
	
	\coordinate (Nroot) at (2 + \tikzSSS ,1);
	\coordinate (Nv1) at (2 + \tikzSSS ,2);
	\coordinate (Nv2r) at (3 + \tikzSSS ,3.5);
	\coordinate (Nv2l) at (1 + \tikzSSS ,3.5);
	\coordinate (Nv3rr) at (3.5 + \tikzSSS ,5);
	\coordinate (Nv3r) at (2.5 + \tikzSSS ,5);
	\coordinate (Nv3l) at (1.5 + \tikzSSS ,5);
	\coordinate (Nv3ll) at (0.5 + \tikzSSS ,5);


	\path (root) edge (v1);
	\path (v1) edge (v2r);
	\path (v1) edge (v2l);
	\path (v2r) edge (v3rr);
	\path (v2r) edge (v3r);
	\path (v2l) edge (v3l);
	\path (v2l) edge (v3ll);

	\node [GZ node] at (root)  {\white$\ell$};
	\node [G node] at (v1) {$j$};
	\node [G node] at (v2r) {$i$};
	\node [G node] at (v2l) {$i'$};
	\node [G node] at (v3rr) {$k$};
	\node [G node] at (v3r) {$k'$};
	\node [G node] at (v3l) {$h$};
	\node [G node] at (v3ll) {$h'$};


	\draw (Nroot) to [R] (Nv1);
	\draw (Nv1) to [R] (Nv2r);
	\draw (Nv1) to [R] (Nv2l);
	\draw (Nv2r) to [ R] (Nv3rr);
	\draw (Nv2r) to [R] (Nv3r);
	\draw (Nv2l) to [R] (Nv3l);
	\draw (Nv2l) to [R] (Nv3ll);
	\draw (7,.4)  node [rectangle,draw]{};
	\draw (7,.4) to [battery] (7,1) to (Nroot){};
	\draw (Nroot) to [-] (5 ,1) to [R] (5 ,.4) node [rectangle,draw]{};
	\draw (Nv1) to [-] (5,2) to [R] (5,1.4) node [rectangle,draw]{};
	\draw (Nv2r) to [-] (8.2 ,3.5) to [R] (8.2 , 2.9) node [rectangle,draw]{};
	\draw (Nv2l) to [-] (4.7,3.5) to [R] (4.7 ,2.9) node [rectangle,draw]{};
	\draw (Nv3rr) to [-] (8.2 ,5) to [R] (8.2,4.4) node [rectangle,draw]{};
	\draw (Nv3r) to [-] (6.4,5) to [R] (6.4,4.4) node [rectangle,draw]{};
	\draw (Nv3l) to [-] (1.8 + \tikzSSS ,5) to [R] (1.8 + \tikzSSS ,4.4) node [rectangle,draw]{};
	\draw (Nv3ll) to [-] (.1 + \tikzSSS ,5) to [R] (.1 + \tikzSSS ,4.4) node [rectangle,draw]{};

	\node [GZ node] at (Nroot)  {\white$\ell$};
	\node [G node] at (Nv1) {$j$};
	\node [G node] at (Nv2r) {$i$};
	\node [G node] at (Nv2l) {$i'$};
	\node [G node] at (Nv3rr) {$k$};
	\node [G node] at (Nv3r) {$k'$};
	\node [G node] at (Nv3l) {$h$};
	\node [G node] at (Nv3ll) {$h'$};

\end{tikzpicture}
\caption{\label{fig:tree-G-N} A tree graph $\GG = (I,E)$ on the left. The electrical network $\NN = (I_\bias, E\cup E_\bias)$  on the right corresponds to the case where every element of the vector $\qq$ is larger than zero. Signal grounds are drawn with empty squares.}
\end{figure}
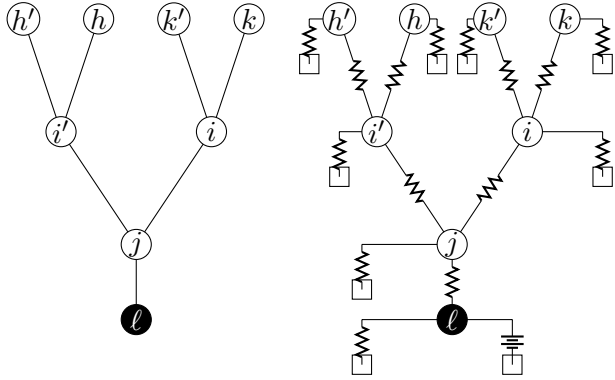

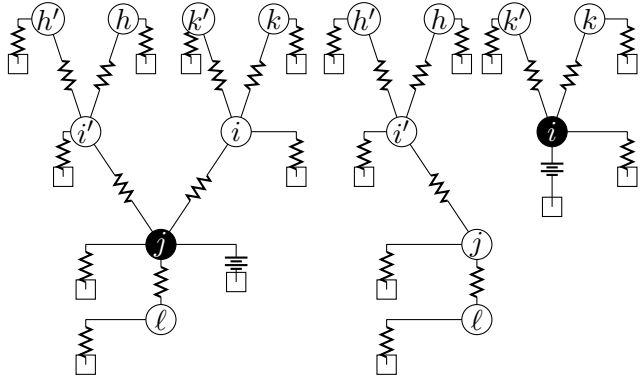
\begin{figure}
\centering
\ctikzset{ bipoles/length =.5 cm , label/align = straight}
\begin{tikzpicture}

	
	\tikzstyle{G node} =  [shape = circle, draw, fill = white, minimum size = 4mm, inner sep = 0mm];
	\tikzstyle{GX node} =  [shape = circle, draw, fill = white, minimum size = 4mm, inner sep = 0mm];
	\tikzstyle{GZ node} = [shape = circle, draw, fill = black, minimum size = 4mm, inner sep = 0mm];

	\coordinate (root) at (2,1);
	\coordinate (v1) at (2,2);
	\coordinate (v2r) at (3,3.5);
	\coordinate (v2l) at (1,3.5);
	\coordinate (v3rr) at (3.5,5);
	\coordinate (v3r) at (2.5,5);
	\coordinate (v3l) at (1.5,5);
	\coordinate (v3ll) at (0.5,5);
	
	\coordinate (Nroot) at (2 + \tikzSSS ,1);
	\coordinate (Nv1) at (2 + \tikzSSS ,2);
	\coordinate (Nv2r) at (3 + \tikzSSS ,3.5);
	\coordinate (Nv2l) at (1 + \tikzSSS ,3.5);
	\coordinate (Nv3rr) at (3.5 + \tikzSSS ,5);
	\coordinate (Nv3r) at (2.5 + \tikzSSS ,5);
	\coordinate (Nv3l) at (1.5 + \tikzSSS ,5);
	\coordinate (Nv3ll) at (0.5 + \tikzSSS ,5);

	\draw (root) to [R] (v1);
	\draw (v1) to [R] (v2r);
	\draw (v1) to [R] (v2l);
	\draw (v2r) to [ R] (v3rr);
	\draw (v2r) to [R] (v3r);
	\draw (v2l) to [R] (v3l);
	\draw (v2l) to [R] (v3ll);
	\draw (root) to [-] (1 ,1) to [R] (1 ,.4) node [rectangle,draw]{};
	\draw (v1) to [-] (1,2) to [R] (1,1.4) node [rectangle,draw]{};
	\draw (v2r) to [-] (3.8 ,3.5) to [R] (3.8 , 2.9) node [rectangle,draw]{};
	\draw (v2l) to [-] (0.7,3.5) to [R] (0.7 ,2.9) node [rectangle,draw]{};
	\draw (v3rr) to [-] (3.8 ,5) to [R] (3.8,4.4) node [rectangle,draw]{};
	\draw (v3r) to [-] (2.4,5) to [R] (2.4,4.4) node [rectangle,draw]{};
	\draw (v3l) to [-] (1.8  ,5) to [R] (1.8 ,4.4) node [rectangle,draw]{};
	\draw (v3ll) to [-] (.1  ,5) to [R] (.1  ,4.4) node [rectangle,draw]{};
	
	\draw (3,1.5) node [rectangle,draw]{};
	\draw (3,1.5) to [battery] (3,2) to (v1);

	\node [G node] at (root)  {$\ell$};
	\node [GZ node] at (v1) {\white$j$};
	\node [G node] at (v2r) {$i$};
	\node [G node] at (v2l) {$i'$};
	\node [G node] at (v3rr) {$k$};
	\node [G node] at (v3r) {$k'$};
	\node [G node] at (v3l) {$h$};
	\node [G node] at (v3ll) {$h'$};

\begin{scope}
\draw (Nroot) to [R] (Nv1);
	\draw (Nv1) to [R] (Nv2l);
	\draw (Nv2l) to [R] (Nv3l);
	\draw (Nv2l) to [R] (Nv3ll);
	\draw (Nroot) to [-] (5 ,1) to [R] (5 ,.4) node [rectangle,draw]{};
	\draw (Nv1) to [-] (5,2) to [R] (5,1.4) node [rectangle,draw]{};
	\draw (Nv2l) to [-] (4.7,3.5) to [R] (4.7 ,2.9) node [rectangle,draw]{};
	\draw (Nv3l) to [-] (1.8 + \tikzSSS ,5) to [R] (1.8 + \tikzSSS ,4.4) node [rectangle,draw]{};
	\draw (Nv3ll) to [-] (.1 + \tikzSSS ,5) to [R] (.1 + \tikzSSS ,4.4) node [rectangle,draw]{};
		
\end{scope}

	\draw (Nv2r) to [ R] (Nv3rr);
	\draw (Nv2r) to [R] (Nv3r);
	\draw (Nv2r) to [-] (8.2 ,3.5) to [R] (8.2 , 2.9) node [rectangle,draw]{};
	\draw (Nv3rr) to [-] (8.2 ,5) to [R] (8.2,4.4) node [rectangle,draw]{};
	\draw (Nv3r) to [-] (6.4,5) to [R] (6.4,4.4) node [rectangle,draw]{};

	\draw (7.2,2.5) node [rectangle,draw]{};
	\draw (7.2,2.5) to [battery] (Nv2r);

	\node [GZ node] at (Nv2r) {\white$i$};

	\node [G node] at (Nv3rr) {$k$};
	\node [G node] at (Nv3r) {$k'$};

	\node [GX node] at (Nroot)  {$\ell$};
	\node [GX node] at (Nv1) {$j$};
		\node [GX node] at (Nv3l) {$h$};
	\node [GX node] at (Nv3ll) {$h'$};
		\node [GX node] at (Nv2l) {$i'$};
\end{tikzpicture}
\caption{\label{fig:def-W-H} Consider an electrical network whose underlying graph $\GG$ is a tree, like in Figure~\ref{fig:tree-G-N}. Then, the quantity $W^{i\to j}(\infty)$ is the potential of $i$ if $j$ has potential one (on the left), while the quantity $H^{i\to j}(\infty)$ is the influence of $i$ in the network where the edge $\{i,j\}$ has been removed (on the right). Signal grounds are drawn with empty squares. }
\end{figure}

Consider a connected graph $\GG = (I,E)$ and the opinion dynamic model with matrix $Q$ and vector $\qq$, such that Assumption~\ref{ass:zero},~\ref{ass:Qspeciale} and~\ref{ass:null-Z-bias} are satisfied. 
Let $t \in\{0,1,\ldots\}$ be the iteration counter. 
 The MPA 
 works as follows.
For any ordered pair of nodes $(j,i)$ such that $\{i,j\} \in E$, the node $i$ sends to its neighbour $j$ two messages, $\Wij(t) \in [0,1]$ and $\Hij(t) \in \Rp$, initialised by:
\begin{align*}
	\Wij(0) = 1 \qquad  \Hij(0) = 1\,.
\end{align*}
All messages are updated synchronously following the rules:
\begin{align}
\label{eq:MPA-update-1}
&\Wij(t+1) = \frac{1}{1 + \frac{ q_i }{Q_{ij}} + {\textstyle \sum_{k \in N_i^j }} \frac{Q_{ik}}{Q_{ij}} \left(1- \Wki(t) \right) } \\ 
&\Hij(t+1) = 1 + {\textstyle \sum_{k \in N_i^j }} \Wki(t)\, \Hki(t)\,, \label{eq:MPA-update-2} 
\end{align} 
where $N_i^j:= N_i \setminus \{j\}$. 
At any time, node $\ell \in I$ can compute an approximation of its influence $H(\ell)$ by:
\begin{align*}
	H^\ell(t)  =  1 + {\textstyle \sum_{i \in N_\ell }} \Wil(t)\, \Hil(t)\,. \nonumber
\end{align*}

If the graph $\GG = (I,E)$ is a tree, the MPA converges in a number of steps equal to the diameter of the graph and is exact, i.e. $ \lim_{t \to \infty }  H^\ell(t) = H(\ell)\,.$ 
The asymptotic values of the messages have a simple and exact interpretation (see  Figure~\ref{fig:tree-G-N} and~\ref{fig:def-W-H}):
\begin{enumerate}[leftmargin=2cm] 
\item[$\Wij(\infty)$~] 
	is the asymptotic opinion of $i$ if the leader is $j$. In other words, is the potential of $i$ in the electrical network where $j$ is connected to a unitary voltage source. 
\item[$\Hij(\infty)$~] 
	is the  influence of $i$ in its sub-graph, after the edge $\{i,j\}$ has been removed. In other words, it is the sum of the node's potentials in the sub-network containing $i$, after the edge $\{i,j\}$ has been removed and 
	$i$ has been connected to a unitary voltage source.
\end{enumerate} 
The electrical interpretation of $\Wij(\infty)$ and $\Hij(\infty)$  allows to recover the rules \eqref{eq:MPA-update-1} and \eqref{eq:MPA-update-2} easily. 
First observe that, using $\Wki(\infty)$ for $k\in N_i^j$,  $\Wij(\infty)$ can be computed by:
$$ \Wij(\infty) = \frac{1}{1 + \frac{ C_{i\bias} }{C_{ij}} + {\textstyle \sum_{k \in N_i^j }} \frac{C_{ik}}{C_{ij}} \left(1- \Wki(\infty) \right) }\,. $$
Each value $\Wki(\infty)$ encodes the equivalent conductance from $k$ to the reference node, in the graph $\GG$ without the edge $\{i,k\}$. Actually, the update rule \eqref{eq:MPA-update-1} comes from the expression above, since $ \frac{ q_i }{Q_{ij}} = \frac{ C_{i\bias} }{C_{ij}} $ and  $  \frac{Q_{ik}}{Q_{ij}} = \frac{C_{ik}}{C_{ij}} $.

Now, consider two edges, $\{i,j\}$ and $\{i,k\}$, with a common node $i$, and let the potential of $j$ be fixed to one. 
Since $\GG$ is a tree, any current to $k$ must flow through $i$. 
Although $\{j,k\} \notin E$, denote by $W^{k\to j}(\infty)$ the potential of $k$ and observe that: 
$$W^{k\to j}(\infty) =  \Wki(\infty) \, \Wij(\infty) \,. $$
The property above 
(see~\cite[Eq. (10)]{Vassio:2014:journal})
allows to compute $H^\ell(\infty)$ recursively: each potential is expressed as product of factors, that can be re-organized following the interpretation of $\Hij(\infty)$. 

If the graph $\GG$ is a tree, the asymptotic values $\Wij(\infty)$ and $\Hij(\infty)$ can be used to compute $H^\ell(\infty)$ by a series of recursive steps that start from the leaves of $\GG$ and continues towards the node $\ell$. Actually, by the MPA any node $\ell$ computes the correct value $H(\ell)$ in a number of steps that is smaller or equal to the diameter of the graph.

More in general, asymptotic convergence can be proved for the algorithm, as stated in the following result that is the main contribution of this paper.

\begin{theorem}[Convergence]\label{thm:conv-MPA-G}
Under the standing assumptions, the MPA converges on any graph $\GG$.
\end{theorem}
Thus, the MPA converges even if the graph $\GG$ is not a tree, but convergence is asymptotical (not in finite time) and the limit values do not in general provide the exact values of the asymptotic influence (that is, $H^\ell(\infty) \neq H(\ell)$).
We shall explore the issues of convergence time and of asymptotical error by simulations in Section~\ref{sec:simulations}, after proving Theorem~\ref{thm:conv-MPA-G} in the next section.

\begin{remark}[Relation with Vassio et al., 2014]
The algorithm \eqref{eq:MPA-update-1}-\eqref{eq:MPA-update-2} is an adaptation of the MPA proposed by~\cite{Vassio:2014:journal,Vassio:2014:ECC} to solve an equivalent problem (as we discussed above). The paper~\cite{Vassio:2014:journal} proves the convergence of the algorithm on graphs such that every node has the same degree (regular graphs). 
Our Theorem~\ref{thm:conv-MPA-G} immediately extends the proof of convergence of the MPA in~\cite{Vassio:2014:journal} to every connected graph, provided the set $S^0$ is non-empty.
Indeed, in the MPA in~\cite{Vassio:2014:journal} the messages flowing out the nodes in $S^0$ are always zero. The nodes in $S^0$ can be pruned from that graph and, correspondingly, such a fixed (and null) message contribution shall be constantly added to the computations carried by the neighbouring nodes. In our setting this corresponds to properly choosing the value of the vector $\qq$.
\end{remark}

\section{Convergence proof}\label{sect:convergence}
This section is devoted to the proof of Theorem~\ref{thm:conv-MPA-G}. To this goal, we shall first study a related non-linear dynamics on directed graphs and prove its convergence in Subsection~\ref{subsec:conv-MPA-like-D}. Then, in Section~\ref{subsec:proof-conv-MPA-G} we recognise that the MPA can be mapped into a special case of this dynamic and thus prove its convergence.

\subsection{Convergence of a MPA-like dynamics on digraphs}\label{subsec:conv-MPA-like-D}  

In this subsection we introduce the \textit{directed} graphs, define a suitable nonlinear dynamic (that generalizes \eqref{eq:MPA-update-1}-\eqref{eq:MPA-update-2} in a certain sense) and prove under which conditions it converges on any digraph. The proof is straightforward for acyclic graphs but more involved for graphs that contain strongly connected components. 

A \textit{directed} graph or \textit{digraph} is a pair $\DD = (V, \Phi )$ where $V$ is the set of vertices and $\Phi \subseteq V \times V$ is the set of \textit{arcs}, that are ordered pairs of vertices. The sub-digraph induced by $U\subseteq V$ is $\DD[U] = (U, \Phi \cap U\times U)$. 
 A node $v$ is a \textit{sink} if $(v,w) \notin \Phi$ for any $w \in V$. An arc of the form $(v,v)$  
is a \textit{self-loop}. 
A \textit{path from $v$ to $w$} on $\DD$, of length $l$, is an ordered list of nodes $(u_0, u_i, \ldots, u_l)$ such that 
\begin{enumerate}
\item[(i)] $u_0 = v$ and $u_l = w$;
\item[(ii)] $(u_{i-1},u_{i}) \in \Phi$ for every $i \in \{1,\ldots, l\}$. 
\end{enumerate}
A path is \textit{simple} if no arc is repeated.  
A node $w$ is \textit{reachable} from $v$ if there exists a path of length $l \ge 0$ from $v$ to $w$.

A digraph $\DD = (V, \Phi )$ 
is termed \textit{strongly connected} if for every pair of nodes $v,w \in V$, there is a path from $v$ to $w$ and a path from $w$ to $v$. 
If $\DD$ is not strongly connected, let $U\subset V$. The induced sub-digraph $\DD[U]$ is a \textit{strongly connected component} of $\DD$ if $\DD[U]$ is strongly connected but $\DD[U\cup\{v\}]$ is not, for any $v \in V\setminus U$.
A strongly connected component $\DD[U]$ is \textit{trivial} if it contains a single node without a self-loop, i.e. $\DD[U] = (\{u\},\emptyset)$. 
Otherwise it is \textit{non-trivial}.  
The digraph $\DD$  
is \textit{acyclic} if all its strongly connected component are trivial.  
We term \textit{acyclic ordering} a relabeling $x_1, x_2, \ldots, x_{|V|}$  of the vertices of $\DD$ such that for every arc $(x_i,x_j) \in \Phi$ it holds $j < i$. Any acyclic digraph admits an acyclic ordering~\cite[Prop 2.1.3]{BJG:2009}.

Given any digraph $\DD = (V,\Phi)$, consider all its strongly connected components $\DD_k = (V_k, \Phi_k)$, $k \in \{1,\ldots,s\}$.
The \textit{condensation digraph} $\CC_\DD$ of $\DD$ is a digraph with vertex set $\{1,\ldots,s\}$ such that there is an arc from $h$ to $k$ if  there is an arc in $\DD$ from a node in $V_h$ to a node in $V_k$ and $k \neq h $. It is easy to check that $\CC_\DD$ is acyclic.

We introduce now the non-linear dynamic of interest, on the digraph $\DD = (V, \Phi )$. 

Let two sequences of non-negative vectors $\aalpha(t)$, $\bbeta(t) \in [0,+\infty)^V$, and two positive vectors $\rr$, $\ss \in (0,+\infty)^V$ be given, which satisfy the following condition.
\begin{assumption}\label{ass:abrs} 
The vectorial sequence $\aalpha(t)$ is non-decreasing in every component and $\bbeta(t) $ is convergent.
The vectors $\rr$ and $\ss$ are such that $r_v = s_v^{-1}$ for every $v\in V$.
\eabullet\end{assumption}
This assumption will hold in the rest of this subsection.

Let $M \in \{0,1\}^{V\times V}$ be the adjacency matrix of the digraph $\DD$, i.e. $M_{vw} = 1$ if and only if $ (v,w) \in \Phi$. Let $W \in [0,+\infty)^{V\times V}$ be a non-negative matrix such that:
 $$W_{vw} = r_v  M_{vw} s_w\,.$$ 
Consider two vector sequences $\oomega(t) \in (0,1]^V$ and $\eeta(t) \in [1,+\infty)^V$ of initial value $\oomega(0) = \eeta(0) = \ones$. 
The dynamic is given by the update rules: 
\begin{align}
	\omega_v(t+1) &= \frac{1}{1 + \alpha_v(t) + \sum_w W_{vw}\, (1 - \omega_w(t))}\,, \label{eq:omega-update} \\[2mm]
	\eta_v(t+1) &= 1 + \beta_v(t) + \textstyle{ \sum_w} M_{vw} \,\omega_w(t) \,\eta_w(t) \,, \label{eq:eta-update}
\end{align}
for every $v\in V$ and $t\ge 0$.

The sequences $\oomega(t)$ and $\eeta(t)$ converge in any acyclic digraph, because the pattern of interdependencies among the components of the sequences follow an acyclic order.  
\begin{lemma}[Convergence--Acyclic digraphs]\label{lem:D-acyclic}  
If the digraph $\DD=(V,\Phi)$ is acyclic, then the sequence $\eeta(t)$ is convergent and 
the sequence $\oomega(t)$ is non-increasing in every component and convergent. Moreover, $\lim_{ t\to +\infty} \omega_v(t) < 1$ 
if and only if there exists $w$ reachable from $v$ such that $\alpha_w(t)$ is non identically zero. 
\end{lemma}
\begin{proof} 
Let the subset $S\subseteq V$ contain the sink nodes of the digraph $\DD$. Since $\DD$ is acyclic, $S$ is non-empty~\cite[Prop 2.1.1]{BJG:2009}. 
If $v \in S$, then $M_{vw}=W_{vw}=0$ for every $w$. For any $v \in S$, the update rules \eqref{eq:omega-update} and \eqref{eq:eta-update} simplify to  $\omega_v(t+1) = \frac{1}{1 + \alpha_v(t)} $ and $\eta_v(t+1) = 1+ \beta_v(t)$. 
 Given Assumption~\ref{ass:abrs}, 
 the sequence $w_v(t)$ is non-increasing  while $\eta_v(t)$ converges.  
Moreover, $\lim \omega_w(t) < 1$ if and only if $\alpha_v(t)$ is non identically zero. 

If $V\setminus S$ is non-empty, i.e. there are non-sink nodes,
we introduce an acyclic ordering $x_1, x_2, \ldots, x_{|V|}$ on $V$ such that $\{ x_1,\ldots, x_{|S|}\} \equiv S$, and  proceed by induction. Let $k \geq 2 $ and assume that, for all $i < k$, $\omega_{x_i}(t)$ is non-increasing while $\eta_{x_i}(t)$ converges. 
Moreover, assume that $\lim \omega_{x_i}(t) <1$ if and only if there exists a $x_j$  reachable from $x_i$ (where $j\leq i$) such that $\alpha_{x_j}(t)$ is non identically zero.  
Since 
$W_{x_k x_i}=0 $  
for any $i \ge k$, the update law \eqref{eq:omega-update} of $\omega_{x_k}(t)$ is equivalent to:
$$\omega_{x_{k}}(t+1) = \frac{1}{1 + \alpha_{x_{k}}(t) + \sum_{i < k} W_{{x_{k}}x_i}\, (1 - \omega_{x_i}(t))}\,.$$
The denominator is the sum of non-decreasing terms, thus $\omega_{x_k}(t)$ is non-increasing and always belongs to $(0,1]$. 
Moreover, $\lim \omega_{x_k}(t) <1$ iff $\alpha_{x_k}(t)$ is non identically zero or there exists $W_{{x_{k}}x_i} >0$ and $\lim \omega_{x_i}(t) <1$. 
Hence, $\lim \omega_{x_k}(t) <1$ iff there exist $x_j$ reachable from $x_k$ and $\alpha_{x_k}(t)$ is non identically zero.

The update law \eqref{eq:eta-update} for $\eta_{x_k}(t)$ simplifies to:
$$\eta_{x_k}(t+1) = 1 + \beta_{x_k}(t) + \textstyle{ \sum_{i<k}} M_{x_k x_i} \,\omega_{x_i}(t) \,\eta_{x_i}(t).$$ 
Similarly, $\eta_{x_k}(t)$ converges because its terms are convergent sequences. 
Then, by induction, the sequences $\oomega(t)$ and $\eeta(t)$ converge. 
\end{proof}

The absence of cycles is not necessary for convergence. The following result shows that, if $\DD$ is strongly connected, then the sequence $\oomega(t)$ also converges. However, in order to prove the convergence of $\eeta(t)$, we will additionally need the presence of a node $w$ where $\alpha_w(t)$ is not identically zero.

\begin{lemma}[Convergence--Strongly connected graphs]\label{lem:D-strongly} 
If the digraph $\DD = (V,\Phi)$ is strongly connected and there exists $v$ such that $\alpha_v(t)$ is not identically zero, then the sequences $\oomega(t)$ and $\eeta(t)$ converge. Moreover, for every $u \in V$ the sequence $\omega_u(t)$ is non-increasing and has limit $\omega_u(\infty)<1$. 
\end{lemma}
\begin{proof} First, we show that $\oomega(t)$ converges to a limit, with every component strictly smaller than 1. Using  the implicit form of the limit we show that the matrix $M \diag(\oomega(t))$ is eventually Shur stable and conclude that  $\eeta(t)$ converges.

Given the assumptions on $\aalpha(t)$, there exists $s \ge 0$ and $v\in V$ such that $\aalpha(t)=\0$ for $t < s$ whereas $\alpha_v(s) > 0$. 
Hence, $\oomega(t) = \1$ for $t \leq s$ whereas $\oomega(s+1) \vecl \1 $, because $\omega_v(s+1) <1$.
We introduce the subset: 
$$R_t :=\{v:\omega_v(t) <1 \}\,, $$
and proceed by induction. Let $t \geq s+1$ and assume that $\oomega(t) \vecl \oomega(t-1)$ and that, unless $R_{t-1} \equiv V$,  $R_{t}$ is a proper superset of $R_{t-1}$.
There are two cases: if $R_t \neq V$, thanks to the strong connectivity, it is always possible to find a pair of nodes $v,w$ such that $v \notin R_t$, $w \in R_t$ and $(v,w) \in \Phi$. 
Then $\omega_v(t+1) <1$  so $v \in R_{t+1}$ and $\oomega(t+1) \vecl \oomega(t)$.

If $R_t = V$,  there is a $w'$ such that $\omega_{w'}(t) < \omega_{w'}(t-1)$ by hypothesis and it always exists a node $v$ such that $(v,w') \in \Phi$, thanks to the strong connectivity. For that $v$ it holds:
\begin{align*}
	\omega_v(t\!&+\!1) = \frac{1}{1 + \alpha_v(t) + \sum_w W_{vw}\, (1\! -\! \omega_w(t))} \\[2mm]
	&< \frac{1}{1 + \alpha_v(t\!-\!1) + \sum_w W_{vw}\, (1\! -\! \omega_w(t\!-\!1))} = \omega_v(t)\,,
\end{align*}
because $\alpha_v(t)$ is non decreasing, $\omega_{w}(t) \leq \omega_{w}(t-1)$ for every $w$ and there exist $w'$ such that $\omega_{w'}(t) < \omega_{w'}(t-1)$ and $W_{vw'}>0$. 
This completes the induction because $R_{s+1} \neq \emptyset = R_s$. 
Therefore, for every $v \in V$ the sequence $\omega_v(t)$ is non-increasing and non identically one, hence it  
admits a limit of value $ \bar\omega_v :=\lim_{t\to +\infty} \omega_v(t) \in (0,1)$.

We now prove that $M \diag(\oomega(t))$ is eventually Shur stable.
By hypothesis, the sequence $\aalpha(t)$ admits a limit $\bar \aalpha \vecg \zeros$. 
The limit $\bar \oomega$ of the recursion \eqref{eq:omega-update} solves, within $(0,1)^V$, the non-linear system:
\begin{align}
\bar \omega_v &= \frac{1}{1 + \bar\alpha_v + \sum_w r_v  M_{vw} s_w \, (1 - \bar\omega_w)}\qquad \forall v \in V\,,\label{eq:omega-bar}
\end{align}
where we used $W_{vw} = r_v  M_{vw} s_w$.
Since the denominators are positive, we rewrite \eqref{eq:omega-bar}  as:
$$\bar \omega_v \left( 1 + \bar\alpha_v + \textstyle{\sum_w} r_v  M_{vw} s_w\, (1 - \bar\omega_w) \right) = 1 \qquad \forall v\in V\,,$$ 
or equivalently: 
$$r_v \textstyle{\sum_w} \bar \omega_v\, M_{vw}s_w\, (1 - \bar\omega_w) = 1 - \bar \omega_v - \bar\alpha_v\, \bar \omega_v \quad \forall v \in V\,,
$$ 
which by Assumption~\ref{ass:abrs} (i.e. $r_v = s_v^{-1}$, $\forall v$) becomes: 
$$ \textstyle{\sum_w} \bar \omega_v\, M_{vw}s_w\, (1 - \bar\omega_w) = s_v(1 - \bar \omega_v) - \bar\alpha_v\, s_v\, \bar \omega_v \quad \forall v \in V\,.
$$ 
By the change of variables 
 $x_v := s_v(1 - \bar\omega_v)$, $c_v := \bar\alpha_v\, s_v \, \bar \omega_v$ and $B_{vw}:= \omega_v\, M_{vw}$ we  obtain: $$\textstyle{\sum_w} B_{vw}\, x_w = x_v - c_v  \qquad \forall v\in V\,,$$ that in vectorial form reads:
\begin{align} \label{eq:eig-B} B \xx = \xx - \cc \,.\end{align} 
In the ``eigenvalue-like'' expression \eqref{eq:eig-B}, the matrix $B = \diag(\bar \oomega) M $ is non-negative and irreducible: every component of $\bar \oomega$ is positive and $M$ is a non-negative matrix adapted to a strongly connected graph, so it is non-negative and irreducible.
Every component of $\xx$ is positive and $\cc \vecg \zeros$, because every component of $\bar \oomega$ belongs to $(0,1)$,  $\bar \alpha \vecg \zeros$, and every component of $\ss$ is positive.

If we multiply \eqref{eq:eig-B} on the left by $B^{|V|-1}$ and iteratively reuse \eqref{eq:eig-B}, we obtain:
$$B^{|V|} \xx = \xx - \textstyle{\sum_{i=0}^{|V|-1}} B^i \, \cc \,.$$
Every element of the matrix  $\sum_{i=0}^{|V|-1} B^i$ is positive, because $B$ is non-negative and irreducible~\cite[Corollary on p.~52]{frG:1959:v2}. 
Therefore, every component of $\textstyle{\sum_{i=0}^{|V|-1}} B^i \, \cc$ is positive and: 
$$( B^{|V|} \xx)_v < x_v  \qquad \forall v\in V\,, $$  
which implies that the spectral radius of $B^{|V|}$ is strictly smaller than one~\cite[Lemma~34.7]{lS:1993:lezioni}, i.e. $\rho(B^{|V|})<1$. Thus, $\rho(B) <1$ and since $B = \diag(\bar\oomega) M$ and $M \diag(\bar\oomega)$ have the same eigenvalues:
$$\rho\left(M \diag(\bar\oomega)\right) <1\,. $$ 
The matrix $M \diag(\oomega(t))$ converges to $M \diag(\bar\oomega)$ hence its spectral radius is eventually smaller than one.

We conclude the proof observing that $\eeta(t)$ converges. In vectorial form, the update law \eqref{eq:eta-update} reads: 
$$\eeta(t+1) = \ones + \bbeta(t) + M \,\diag(\oomega(t))  \,\eeta(t) \,,$$
where the sequences $\bbeta(t)$ converge.
\end{proof}

If $\DD$ is strongly connected, the presence of at least a node $v$ where $\alpha_v(t)$ is non identically zero is necessary. Else $\oomega(t) = \ones$ for every $t\geq 0$, $\rho(M\diag(\oomega(t))) = \rho(M) \geq 1$ since $M$ is irreducible and  $\eeta(t)$ will grow unboundedly.



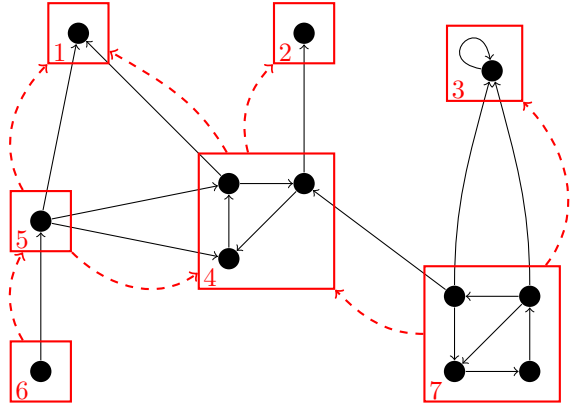
\begin{figure}\centering
\begin{tikzpicture}[scale = 1]

	\tikzstyle{D node} = [shape = circle, fill=black , minimum size = 1mm, inner sep = 1mm];
	\tikzstyle{C node} = [shape = rectangle, draw, red, thick , minimum size = 4mm, inner sep = 4mm];
	\tikzstyle{C label} = [shape = rectangle, red, thick , minimum size = 2mm, inner sep = 2mm];
	
	
	\node[D node] (D1) at (0.5,0.5) {};
	\node[D node] (D2) at (1,5) {};	
	\node[D node] (D3) at (.5,2.5) {};
	\node[D node] (D4) at (3,3) {};
	\node[D node] (D5) at (3,2) {};
	\node[D node] (D6) at (4,3) {};
	\node[D node] (D7) at (6,.5) {};
	\node[D node] (D8) at (6,1.5) {};
	\node[D node] (D9) at (7,.5) {};
	\node[D node] (D10) at (7,1.5) {};
	\node[D node] (D11) at (6.5,4.5) {};
	\node[D node] (D12) at (4,5) {};
	
	\path (D11) edge [->, loop, out = 170, in = 100, looseness = 10] (D11);
	\path (D1) edge [->] (D3);
	\path (D3) edge [->] (D2);
	\path (D3) edge [->] (D5);
	\path (D3) edge [->] (D4);
	\path (D5) edge [->] (D4);
	\path (D4) edge [->] (D6);
	\path (D6) edge [->] (D5);
	\path (D6) edge [->] (D12);
	\path (D4) edge [->] (D2);
	\path (D8) edge [->] (D6);
	\path (D8) edge [->, out = 90, in=-105](D11);
	\path (D10) edge [->,out = 90, in=-75] (D11);
	\path (D10) edge [->] (D8);
	\path (D8) edge [->] (D7);
	\path (D7) edge [->] (D9);
	\path (D9) edge [->] (D10);
	\path (D10) edge [->] (D7);
	
	\node[C node] (C1) at (D1) {};
	\node[C node] (C3) at (D3) {};
	\node[C node] (C4) at (D12) {};
	\node[C node, minimum size = 5mm, inner sep = 5mm] (C5) at (6.4,4.6) {};
	\node[C node] (C2) at  (D2) {};
	\node[C node, minimum size = 9mm, inner sep = 9mm] (C6) at (3.5,2.5) {};
	\node[C node, minimum size = 9mm, inner sep = 9mm] (C7) at (6.5,1) {};
	
	\path (C1) edge [->, red, thick, dashed, out = 120, in = -120] (C3);	
	\path (C3) edge [->, red, thick, dashed, out = 120, in = 225] (C2);	
	\path (C3) edge [->, red, thick, dashed, out = -45, in = 217] (C6);	
	\path (C7) edge [->, red, thick, dashed , out = 180, in = -45] (C6);	
	\path (C7) edge [->, red, thick, dashed, out = 52, in = -45] (C5);	
	\path (C6) edge [->, red, thick, dashed, out = 105, in = 225] (C4);
	\path (C6) edge [->, red, thick, dashed, out = 120, in = -30] (C2);		
	
	\node[C label] at (0.75, 4.75)  {\small$1$};
	\node[C label] at (3.75,4.75)  {\small$2$};
	\node[C label] at (6.05,4.25)  {\small$3$};
	\node[C label] at (2.75,1.75)  {\small$4$};
	\node[C label] at (0.25,2.25)  {\small$5$};
	\node[C label] at (0.25,0.25)  {\small$6$};
	\node[C label] at (5.75,0.25)  {\small$7$};

\end{tikzpicture}
\caption{A digraph $\DD = (V,\Phi)$ and its  condensation digraph $\CC_\DD$.  The digraph $\DD$ has black round nodes and thin arcs. The condensation digraph $\CC_\DD$ has box nodes and dashed edges. The numbers in the nodes of $\CC_\DD$ form an acyclic order. \label{fig:cond-graph}}
\end{figure}

Given a digraph $\DD$, the sequences $\oomega(t)$ and $\eeta(t)$ converge provided that, for any node in a strongly connected component, there exists a reachable $w$ such that $\alpha_w(t)$ is non identically zero. 
To prove the statement, we consider the condensation graph $\CC_\DD$ of $\DD$ and fix an acyclic order on it, as in Figure~\ref{fig:cond-graph}. In any strongly connected component (trivial or not), the dynamic will converge following the acyclic order, similarly to the acyclic case. 
The sequences $\aalpha(t)$ and $\bbeta(t)$, introduced to define the dynamic in \eqref{eq:omega-update} and \eqref{eq:eta-update}, serve to ``connect'' the different components.

\begin{theorem}[Convergence--General graphs] \label{thm:conv-D-general}
Consider dynamics \eqref{eq:omega-update}-\eqref{eq:eta-update} on digraph $\DD = (V, \Phi)$ and recall that Assumption~\ref{ass:abrs} holds.
Assume that, for every node $v$ that belongs to a non-trivial strongly connected component of $\DD$, there exists a node $w$ reachable from $v$ such that the sequence $\alpha_w(t)$ is non identically zero. 
Then, the sequence $\eeta(t)$ converges and the sequence $\oomega(t)$ converges and is non-increasing in every component.
Moreover, 
$\lim_{ t\to +\infty} \omega_v(t) <1$ for every node $v$ such that there exists a $w$ reachable from $v$ and $\alpha_w(t)$ is not identically zero. 
\end{theorem}
\begin{proof} 
Consider the condensation graph $\CC_\DD$ of $\DD$. 
Let $\{1,2\ldots,s\}$ be the vertex set of $\CC_\DD$ and assign the nodes' label to form an   acyclic order on $\CC_\DD$ where the smallest number are reserved to sink nodes, c.f. Figure~\ref{fig:cond-graph}. 
Assume $k$ is the node of $\CC_\DD$ that represents the strongly connected component $\DD_k = (V_k,\Phi_k)$. For every $v \in V_k$, we rewrite the update laws \eqref{eq:omega-update} and \eqref{eq:eta-update} as: 
\begin{align}
	\omega_v(t+1) &= \frac{1}{1 + \alpha'_v(t) + \sum_{w\in V_k} W_{vw}\, (1 - \omega_w(t))}\,, \label{eq:omega-update-rew} \\[2mm]
	\eta_v(t+1) &= 1 + \beta'_v(t) + \textstyle{ \sum_{w \in V_k}} M_{vw} \,\omega_w(t) \,\eta_w(t) \,, \label{eq:eta-update-rew}
\intertext{ for every $t\ge 0$, where:}
	\alpha_v'(t) &:=  \alpha_v(t) + \textstyle{ \sum_{w \notin V_k}} W_{vw}\, (1 - \omega_w(t)) \,,  \label{eq:alpha-p} \\[2mm]
	\beta_v'(t) &:= \beta_v(t) +  \textstyle{ \sum_{w \notin V_k}} M_{vw} \,\omega_w(t) \,\eta_w(t) \label{eq:beta-p} \,.
\end{align}

Let $k$ be a sink node of $\CC_\DD$ (there must be at least one) and observe that $M_{v,w} = W_{v,w} = 0$, for any $v \in V_k$ and $w \notin V_k$. Hence, for any $v \in V_k$ and $t\geq 0$, $\alpha_v'(t) = \alpha_v(t)$ and $\eta'_v(t) = \eta_v(t)$ so the dynamics within the component $\DD_k$ is independent of any other component.
For any $v\in V_k$ the sequences $\omega_v(t)$ and $\eta_v(t)$ converge: if $\DD_k$ is a non-trivial strongly connected component, invoke Lemma~\ref{lem:D-strongly}; else $\DD_k = (\{v\},\emptyset)$ and it is sufficient to observe the expressions, similar to those in the proof of Lemma~\ref{lem:D-acyclic}. Moreover, $\omega_v(t)$ is non-increasing and $\lim \omega_v(t) <1$ if there is $w \in V_k$ such that $\alpha_w(t)$ is non identically zero.

Now, consider any non-sink node $k>1$ of $\CC_\DD$ and assume that the sequences $\omega_u(t)$ and $\eta_u(t)$ converge for any node $u \in V_h $ in any component $\DD_h$ where $h<k$. Assume moreover that $\lim \omega_u(t) <1$ if there exist $w$ reachable from $u$ such that $\alpha_w(t)$ is non identically zero. Let $v \in V_k$ and observe that the sequence $\alpha_v'(t)$ and $\eta_v'(t)$, defined in \eqref{eq:alpha-p} and \eqref{eq:beta-p}, only contain terms $\omega_u(t)$ and $\eta_u(t)$ where $u \in V_h$ for some $h <k$. Given these assumptions, $\alpha_v'(t)$ is non-decreasing, and is non identically zero if there exist a $w$ reachable (in $\DD$) from $v$ such that $\alpha_w(t)$ is non identically zero. 
Moreover $\beta_v'(t)$ converges.
Therefore, thanks to Lemma~\ref{lem:D-strongly} if $\DD_k$ is non-trivial or by observing the expressions if $\DD_k$ is trivial, the sequences $\omega_v(t)$ and $\eta_v(t)$ converge for any $v\in V_k$,  $\omega_v(t)$ is non-increasing and $\lim \omega_v(t) <1$ if there exists $w$ reachable in $\DD$ from $v$ such that $\alpha_w(t)$ is non identically zero. 
An induction on the remaining components of $\CC_\DD$ proves the claim.  
\end{proof}

\subsection{Convergence of the MPA on $\GG$}\label{subsec:proof-conv-MPA-G}

In this subsection we prove the convergence of the MPA on any graph $\GG = (I,E)$ under our standing 
Assumptions~\ref{ass:zero},~\ref{ass:Qspeciale}, and~\ref{ass:null-Z-bias}. 
First we introduce the \textit{message digraph} $\MM_\GG$ corresponding to the graph $\GG$, that describes the topology of the interdependencies between the messages of the MPA, and rewrite the MPA dynamic. Then we show that the assumptions of Lemma~\ref{lem:D-strongly} and Theorem~\ref{thm:conv-D-general} hold, so we can use the results of the previous section. Finally, we show that Assumption~\ref{ass:abrs} holds and prove Theorem~\ref{thm:conv-MPA-G}.

The \textit{message digraph} associated to the graph $\GG = (I,E)$ is the digraph $\MM_\GG = (\vec E, \Phi)$ whose node set $\vec E \subseteq I\times I $ contains the \textit{ordered} pairs of vertices in $I$ that share an edge in $\GG$:
$$\vec E := \{ ji : \{i,j\} \in E  \} \,.$$ 
We reserve the shorthand notation $ji := (j,i)$ for the elements of $\vec E$. 
The arc set of $\MM_\GG$ is the subset $\Phi \subseteq \vec E \times \vec E$ of ordered pairs of $\vec E$ defined by:
$$\Phi := \{(ji,hk) : \{i,j\}\text{~and~} \{k,h\} \in  E,~  i=h,~ j\neq k \} \,.$$
As a simple example, the message digraph $\MM_\GG$ of a line graph $\GG$ with three nodes is shown in Figure~\ref{fig:M-G-topology}: the figure also shows, more in general, how two edges of $\GG$ with a common node map into $\MM_\GG$.
A complete analysis of the topological properties of the message digraph $\MM_\GG = (\vec E, \Phi)$ is out of the scope of this paper.
We just observe that, if  $\GG = (I,E)$ is connected, $\MM_\GG$ enjoys the following properties: 
\begin{itemize}
\item if $\GG$ is a tree, i.e. $|E| = |I|-1$,  $\MM_\GG$ is acyclic;
\item if $\GG$ contains exactly one circuit, i.e. $|E| = |I|$,  $\MM_\GG$ contains exactly two non-trivial strongly connected components;
\item if $\GG$ contains more that one circuit, i.e. $|E| > |I|$,  $\MM_\GG$ contains exactly one non-trivial strongly connected components.
\end{itemize}

\begin{figure}\centering
\begin{tikzpicture}[scale = 1]

	\tikzstyle{D node} = [shape = circle, fill=black , minimum size = 5mm, inner sep = 0mm];
	\tikzstyle{G node} =  [shape = circle, draw, thick , minimum size = 8mm, inner sep = 0mm];


	
	\node at (0,4) {\Large$\mathcal{G}$};
	\node[G node] (j) at (1,4) {\Large$j$};
	\node[G node] (i) at (4,4) {\Large$i$};
	\node[G node] (k) at (7,4) {\Large$k$};
	
	\path (j) edge [thick] node[above] {\large$\{i,j\}$} (i);
	\path (i) edge [thick]  node[above] {\large$\{i,k\}$} (k);
	
	\node at (0,1.5) {\Large$\mathcal{M}_\mathcal{G}$};
	\node[D node] (ji) at (2.5,2) {\color{white}$ji$};
	\node[D node] (ij) at (2.5,1) {\color{white}$ij$};	
	\node[D node] (ik) at (5.5,2) {\color{white}$ik$};	
	\node[D node] (ki) at (5.5,1) {\color{white}$ki$};
	\path (ji) edge [->] node[above] {$(ji,ik)$} (ik);
	\path (ki) edge [->] node[above] {$(ki,ij)$} (ij);
	
\end{tikzpicture}
\caption{\label{fig:M-G-topology} A connected graph $\GG$ that contains 3 nodes and 2 edges (above) and the corresponding message digraph $\MM_\GG$ (below). If the above are just two edge, with a common node, of a larger graph $\GG$, the figure would represents in which elements and arcs of $\MM_\GG$ they transform. }
\end{figure}
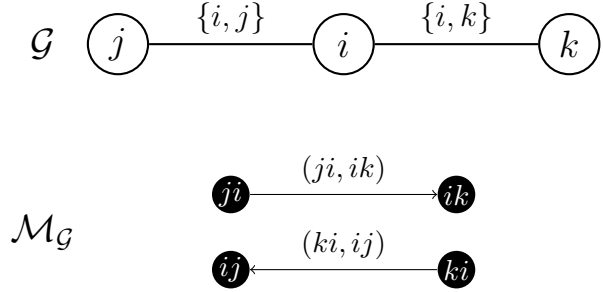

We associate the messages of the MPA to the nodes of $\MM_\GG$.
\begin{quote}
 The messages $\Wij(t)$ and $\Hij(t)$, i.e. those flowing in the edge $\{i,j\}$ of $\GG$ \textit{from $i$ to $j$}, are associated to the element $ji \in \vec E$.
\end{quote}
Following the MPA's update rules \eqref{eq:MPA-update-1} and \eqref{eq:MPA-update-2}, the messages $\Wij(t)$ and $\Hij(t)$ depend on the messages $\Wki(t)$ and $\Hki(t)$ where $k \in N_i \setminus \{ j\}$: the arc $(ji,ik) \in \Phi$ represent such dependence relation.
Note that $\MM_\GG$ does not contain arcs of the form $(ji,ij)$, nodes like $ii$, and self-loops like $(ji,ji)$.

\begin{lemma}\label{lem:G-gamma-ok-for-D-alpha}
Consider a connected graph $\GG=(I,E)$, a matrix $Q \in \Rp^{I\times I}$ adapted to it and a vector $\qq \in \Rp^I$ such that $\qq \vecg \0$. 
Consider the massage digraph $\MM_\GG = (\vec E, \Phi)$ and the vector $\aalpha \in \Rp^{\vec E}$ such that $ \alpha_{hk} = q_k / Q_{kh} $, $\forall \, hk \in \vec E $.
Then, for every $ji$ in a non-trivial strongly connected component of $\MM_\GG$, there exists $hk$ reachable from $ji$ such that $\alpha_{hk}>0$. 
\end{lemma}
\begin{proof} 
Let $ji$ be one node of a non-trivial strongly connected component of $\MM_\GG$.
If $q_i>0$, $\alpha_{ji}>0$ and we are done, so assume $q_i = \alpha_{ji} = 0$.
Notice that $(ji,ji)$ and $(ji,ij) \notin \Phi$ so the strongly connected component (and $\MM_\GG$ itself) must encompass a simple path from $ji$ to itself whose length is at least 3.

Hence, there exists in $\GG$ a circuit that contains the edge $\{i,j\}$. The graph $\GG - \{i,j\}$ (the graph obtained from $\GG$ by the removal of the edge $\{i,j\}$) is still connected.
Let $k \neq i$ be a node of $I$ where $q_k>0$. 
Since $\GG - \{i,j\}$ is connected, it contains a simple path joining $i$ and $k$, of length at least 1. Let $(i,\ldots,h,k)$ be such path. 
Then, the message digraph $\MM_\GG$ contains the path $(ji,\ldots,hk)$ from $ji$ to $hk$, so there exists $hk$ reachable from $ji$ where $\alpha_{hk} = q_k / Q_{kh} > 0$.  
\end{proof}

\begin{figure}\centering
\begin{tikzpicture}[scale = 1]

	\tikzstyle{D node} = [shape = circle, fill=black , minimum size = 5mm, inner sep = 0mm];
	\tikzstyle{G node} =  [shape = circle, draw, thick , minimum size = 8mm, inner sep = 0mm];


	
	\node at (0,4) {\Large$\mathcal{G}$};
	\node[G node] (j) at (1,4) {\Large$j$};
	\node[G node] (i) at (4,4) {\Large$i$};
	\node[G node] (k) at (7,4) {\Large$k$};
	
	\path (j) edge [thick] (i); 
	\path (i) edge [thick] (k); 
	\path (j) edge [thick, dashed, red , <-, out = 20, in = 160] 
	node [above] {$W^{i\to j}(t)$} (i);
	\path (i) edge [thick, dashed, red , <-, out = 20, in = 160] 
	node [above] {$W^{k\to i}(t)$} (k);
	\node at (4.5,3.5) {\color{red}$q_i$};
	\node at (7.5,3.5) {\color{red}$q_k$};
	
	\node at (0,2) {\Large$\mathcal{M}_\mathcal{G}$};
	\node[D node] (ji) at (2.5,2) {\color{white}$ji$};
	\node[D node] (ik) at (5.5,2) {\color{white}$ik$};	
	\path (ji) edge [->] (ik); 
	\node at (2.5,2.6) {\color{red}$\omega_{ji}(t)$};
	\node at (5.5,2.6) {\color{red}$\omega_{ik}(t)$};
	\node at (3,1.5) {\color{red}$\alpha_{ji}$};
	\node at (6,1.5) {\color{red}$\alpha_{ki}$};

\end{tikzpicture}
\caption{\label{M-G-variables} The figure represents the association of the $W$-messages and the vector $\qq$ in $\GG$ with the $\oomega$-variables and the vector $\aalpha$ in $\MM_\GG$. We recall that $\omega_{ij}(t) = W^{i\to j}(t)$ while $\alpha_{ji} = q_i / Q_{ij}$ for every $ji \in \vec E$.}
\end{figure}
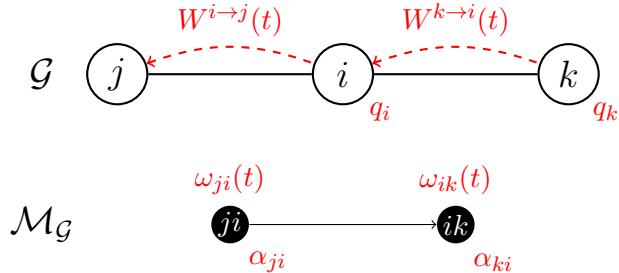

With this lemma about the structure of $\MM_\GG$, we are ready to prove Theorem~\ref{thm:conv-MPA-G} by applying Theorem~\ref{thm:conv-D-general}: this requires to verify that Assumption~\ref{ass:abrs} is satisfied.

\begin{proof}[Proof of Theorem~\ref{thm:conv-MPA-G}]
Consider the message digraph $\MM_\GG = (\vec E , \Phi)$ corresponding to the connected graph $\GG = (I,E)$.
For $t\geq 0$, consider the vector sequences $\aalpha(t),\, \bbeta(t) \in \Rp^{\vec E}$ such that:
\begin{align*}
& \forall ji \in \vec E \quad \alpha_{ji}(t) = q_i/ Q_{ij} \,, \\
& \bbeta(t) =\zeros \,.
\end{align*}
We introduce two vectors $\rr, \ss \in (0,+\infty)^{\vec E}$ such that:
\begin{align*}
	\forall ji \in \vec E \qquad r_{ji} = (C_{ij})^{-1}\,,\quad s_{ji} = C_{ji}\,, 
\end{align*} 
where $C$ is the symmetric matrix in Assumption~\ref{ass:Qspeciale}. 
Since
$s_{ji} = C_{ji} = C_{ij} = r_{ji}^{-1},$ 
the constant vector sequences $\aalpha(t)$, $\bbeta(t)$ and the vectors $\rr$, $\ss$ comply with Assumption~\ref{ass:abrs}. Moreover, we observe that:
$$\frac{Q_{ik}}{Q_{ij}} = \frac{C_{ik}}{C_{ij}} = r_{ji}\, s_{ik}\,, $$
which allows us to rewrite
we rewrite the update law \eqref{eq:MPA-update-1} as: 
$$\Wij(t+1) = \frac{1}{1 + \alpha_{ji} + {\textstyle \sum_{k \in N_i^j }} r_{ji}\, s_{ik} \left(1- \Wki(t) \right) } \,. $$

On the message digraph $\MM_\GG$, consider the dynamic introduced in Subsection~\ref{subsec:conv-MPA-like-D}, for the vector sequences $\oomega(t) \in (0,1]^{\vec E}$ and $\eeta(t) \in [1, +\infty)^{\vec E}$.
Recall the association of the messages $\Wij(t)$ and $\Hij(t)$ to the element $ji \in \vec E$, c.f. Figure~\ref{M-G-variables}. 
 We recognize that, for every $ji \in \vec E$ and $t\geq 0$, it holds:
$$\omega_{ji}(t) \equiv \Wij(t)\, \qquad \eta_{ji}(t) \equiv \Hij(t)\,.$$

According to Assumption~\ref{ass:zero}, $\qq \vecg \zeros$ and  Lemma~\ref{lem:G-gamma-ok-for-D-alpha} guarantees the hypothesis of Theorem~\ref{thm:conv-D-general}.  Then, the messages of the MPA on $\GG$ converge and the sequence $H^\ell(t) $ converges for any node $\ell \in I$.   
\end{proof}

\section{Simulations on random graphs}\label{sec:simulations}

\begin{figure} \centering \fbox{%
	\includegraphics[ 
	trim={25mm} {95mm} {25mm} {80mm}, 
	clip, width={\figwidth}, keepaspectratio=true]{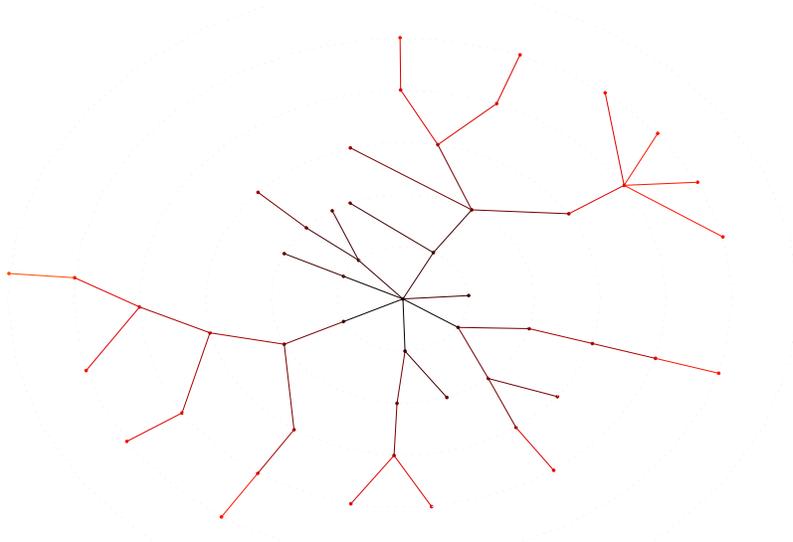}%
	}
	\caption{\label{fig:ST-graph}Radial plot of the spanning tree graph $\GG_{ST} = (I,E_{ST})$ that contains $50$ nodes, $49$ edges, and has diameter 11.}
\end{figure}

\begin{figure}	\centering 	\fbox{%
	\includegraphics[trim={\figtriml} {\figtrimb} {\figtrimr} {\figtrimt}, 	clip, width={\figwidth}, keepaspectratio=true]{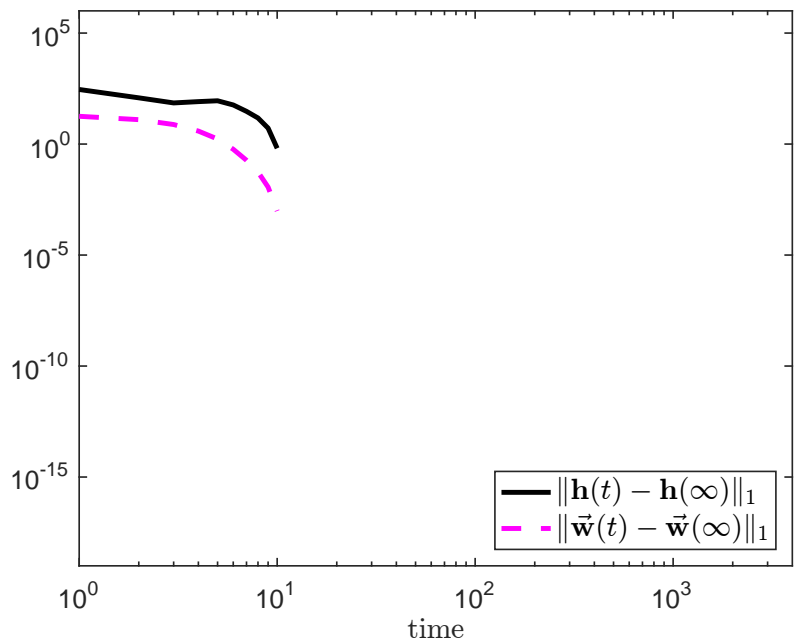}%
	}
	\caption{\label{fig:ST-dyn}The plot shows the MPA convergence on the tree graph $\GG_{ST}$: the solid black line is the error between the vector $\hh(t)$, which contains the node's harmonic influence values, and its limit $\hh(\infty)$; the dashed magenta line is the error between the vector $\vec\ww(t)$, which contains the messages $\Wij(t)$, and its limit $\vec\ww(\infty)$. The algorithm converges in 11 steps.}
\end{figure}

\begin{figure} 	\centering 	\fbox{%
	\includegraphics[trim={\figtriml} {\figtrimb} {\figtrimr} {\figtrimt}, 	clip, width={\figwidth}, keepaspectratio=true]{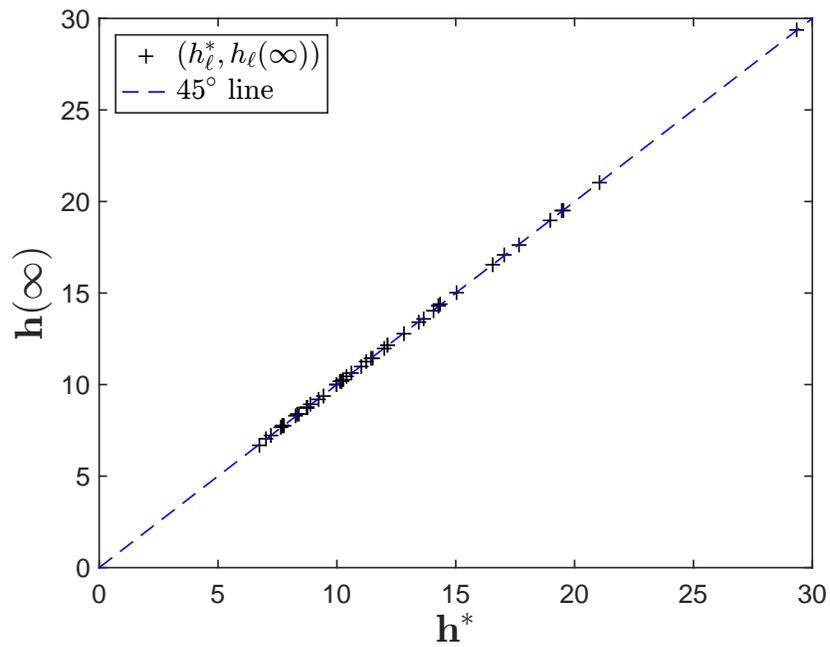}%
	}
	\caption{\label{fig:ST-HH}The elements of the vector $\hh(\infty)$ against the corresponding exact values of the harmonic influence, collected in the vector $\hh^*$, for the graph $\GG_{ST}$. The points are all  aligned on the $45^\circ$ line because the MPA is exact on trees.}
\end{figure}

\begin{figure} 	\centering 	\fbox{%
	\includegraphics[trim={\figtriml} {\figtrimb} {\figtrimr} {\figtrimt}, 	clip, width={\figwidth}, keepaspectratio=true]{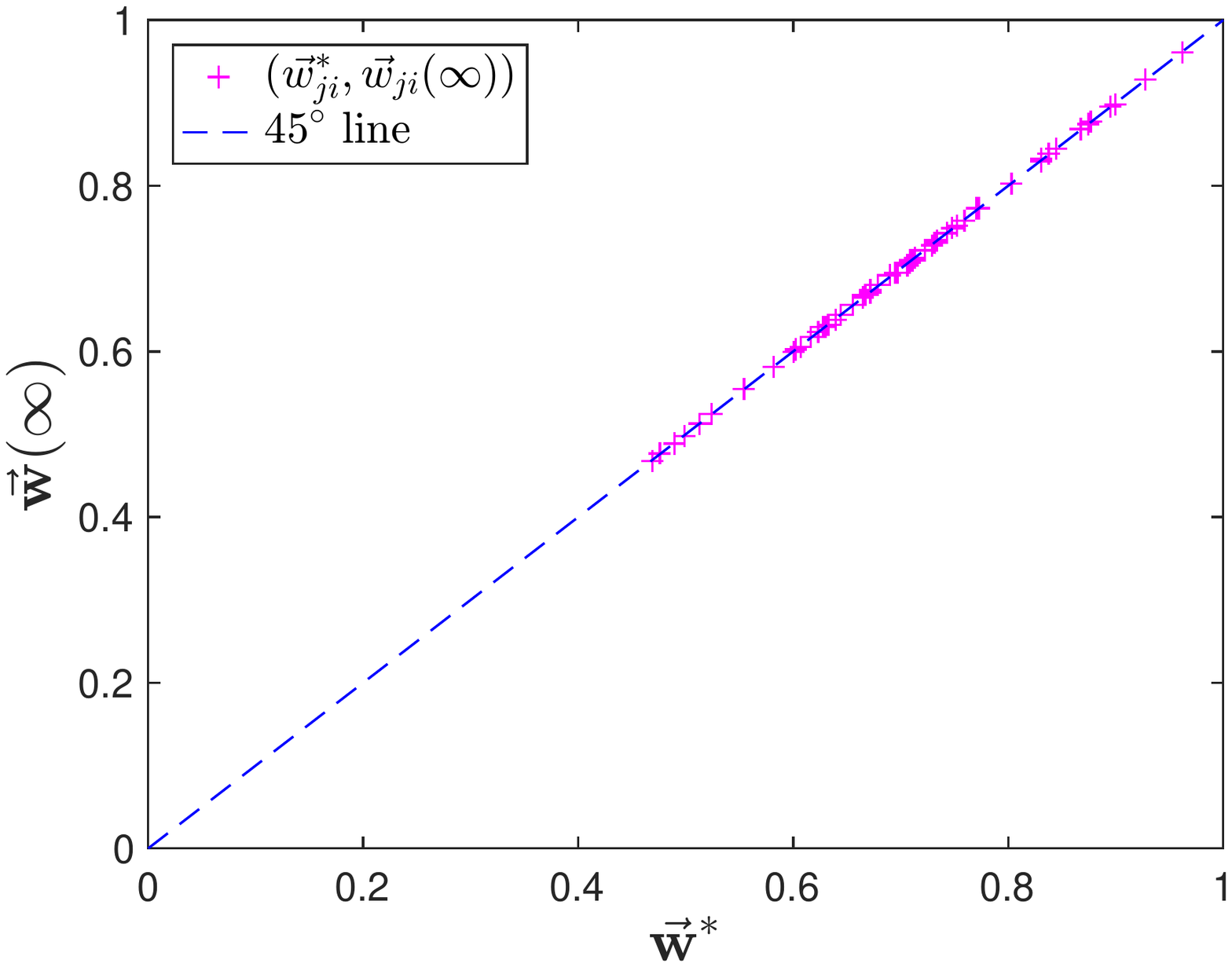}%
	}
	\caption{\label{fig:ST-ww}The elements of the vector $\vec\ww(\infty)$ against the corresponding exact values of the electric interpretation, collected in the vector $\vec\ww^*$, for the graph $\GG_{ST}$. All points are aligned on the $45^\circ$ line because the MPA is exact on trees.}
\end{figure}

Finally, we discuss a few simulations regarding the approximation of the exact \textit{harmonic influence} achieved  by the MPA and its speed of convergence. 

We recall that, given a connected graph $\GG = (I,E)$, the corresponding electrical network is $(\NN, C)$ has node set $I_\bias = I \cup \{\bias\}$, electrical graph $\NN = (I_\bias , E \cup E_\bias)$ and symmetric conductance matrix $C \in \Rp^{I_\bias \times I_\bias}$. In the simulations we assume that the opinion field affects every node, so:
$$ E_\bias = \{ \{i,\bias\}: i\in I \}.$$
The conductance matrix $C$ describes the strength of the influence between agents and by the opinion field. 
For the simulations we assume:
\begin{align*}
	& C_{ij} = 1 \quad \forall \{i,j\}\in E\,, \\
	& C_{i\bias} = \gamma \quad \forall i\in I\,,
\end{align*}
where $\gamma = 0.040$. 

In the simulations we consider three different graphs $\GG$, all with $n = 50$ nodes, but different number of edges. We proceed as follows. 
First we generate an \textit{Erd\H{o}s-R\'enyi} random graph $\GG_{ER} = (I, E_{ER})$ with link probability 0.100. Then, we extract from $\GG_{ER}$ a \textit{spanning tree} $\GG_{ST} = (I, E_{ST})$ and finally we reintroduce \textit{a few} edges to obtain the graph $\GG_{FE} = (I, E_{FE} )$. 

To present the results of the simulations, we introduce the vectors $\hh(t) \in [1,+\infty)^I$ and $\vec{\ww}(t) \in [0,1]^{\vec E}$ and  we use them to stack the approximate harmonic influence $H^\ell(t)$ and  the messages $\Wij(t)$ of the MPA algorithm:
$$ h_\ell(t) = H^\ell(t) \qquad \vec{w}_{ji}(t) = \Wij(t) \,.$$
We run the MPA on each graph and compute the 1-norm errors:
$$\|\hh(t) - \hh(\infty) \|_1\,, \qquad \| \vec\ww(t) - \vec\ww(\infty) \|_1\,, $$
where $ \hh(\infty)$ and $ \vec\ww(\infty)$ denote the convergence values of $\hh(t)$ and $\vec\ww(t)$, respectively. We use these quantities to check the speed of convergence of the MPA.
For reference, we compute the exact harmonic influence, and collect the results in the vector $\hh^* \in [1, +\infty)^I $ such that $ h^*_\ell = H(\ell) $.
Using the electrical interpretation, we also compute the exact value of $\Wij(\infty)$ (i.e. the potential of $i$ in the network $(\NN, C)$ where $j$ is connected to a unitary voltage source). We collect these potentials in the vector $\vec\ww^*$.
To asses the approximation achieved by the MPA, we will compare the asymptotic values $\hh(\infty)$ and $\vec\ww(\infty)$ with the exact values $\hh^*$ and $\vec\ww^*$, and compute the Spearman's rank-order correlation coefficient~\cite{spearman}.

Since the edge sets of the three graphs satisfy  $E_{ST} \subset E_{FE} \subset E_{ER}$, we first discuss the results for the spanning tree $\GG_{ST}$, then for the graph $\GG_{FE}$ and finally for the Erd\H{o}s-R\'enyi graph $\GG_{ER}$. We use the results for the spanning tree to introduce the graphical representation.


The Message Passing Algorithm (MPA) is exact on trees and this is confirmed by the simulations on the tree graph $\GG_{ST}$, visible in Figure~\ref{fig:ST-graph}. The graph $\GG_{ST}$ is a spanning trees of $\GG_{ER}$ (see Figure~\ref{fig:ER-graph}); it has 49 edges and the diameter (i.e. the length of the longest path) is~11. 
The plot in Figure~\ref{fig:ST-dyn} represents the convergence of the MPA. The solid black line is the error between the vector $\hh(t)$ and its limit $\hh(\infty)$; the dashed magenta line is the error between the vector $\vec\ww(t)$ and its limit $\vec\ww(\infty)$. The last point on the plot is at $t=10$: the algorithm converges in 11 steps because the graph $\GG_{ST}$ is a finite tree of diameter 11.   

To discuss the approximation achieved by the MPA, we compare the limit values of the harmonic influence and of the $W$-messages with their corresponding exact values. 
The black crosses in Figure~\ref{fig:ST-HH} represent the pairs:
$$(h^*_\ell, h_\ell(\infty) )\,,$$
where $h^*_\ell = H(\ell)$ is the exact value of the harmonic influence of $\ell \in I$ while $h_\ell(\infty) = H^\ell(\infty)$ is the  approximation computed by the MPA. Since the MPA is exact on trees, all the points are aligned on the $45^\circ$ line. 
Each magenta cross in Figure~\ref{fig:ST-ww} is a pair 
$(\vec w_{ji}^*, \vec w_{ji} (\infty) ) $, where $\vec w_{ji}^*$ is the potential of $i$ if $j$ is connected to a unitary voltage source, while  $\vec w_{ji} (\infty)$ is the limit value taken by the message $\Wij(t)$. 
Again, since the MPA is exact on the tree graph $\GG_{ST}$,  the points are all  aligned on the $45^\circ$.

\begin{figure} \centering \fbox{%
	\includegraphics[ 
	trim={28mm} {90mm} {22mm} {83mm}, 
	clip, width={\figwidth}, keepaspectratio=true]{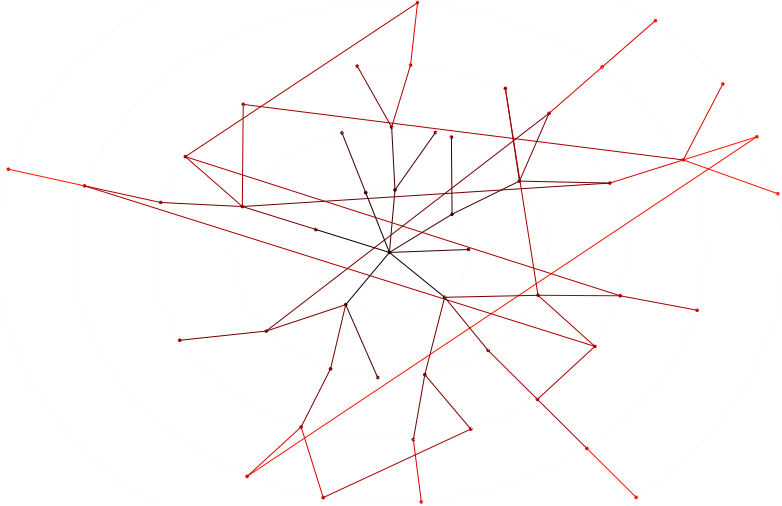}%
	}
	\caption{\label{fig:FE-graph}Radial plot of the graph $\GG_{FE} = (I, E_{FE})$ that contains $50$ nodes, $59$ edges, and has diameter 9.}
\end{figure}

\begin{figure}	\centering 	\fbox{%
	\includegraphics[trim={\figtriml} {\figtrimb} {\figtrimr} {\figtrimt}, 	clip, width={\figwidth}, keepaspectratio=true]{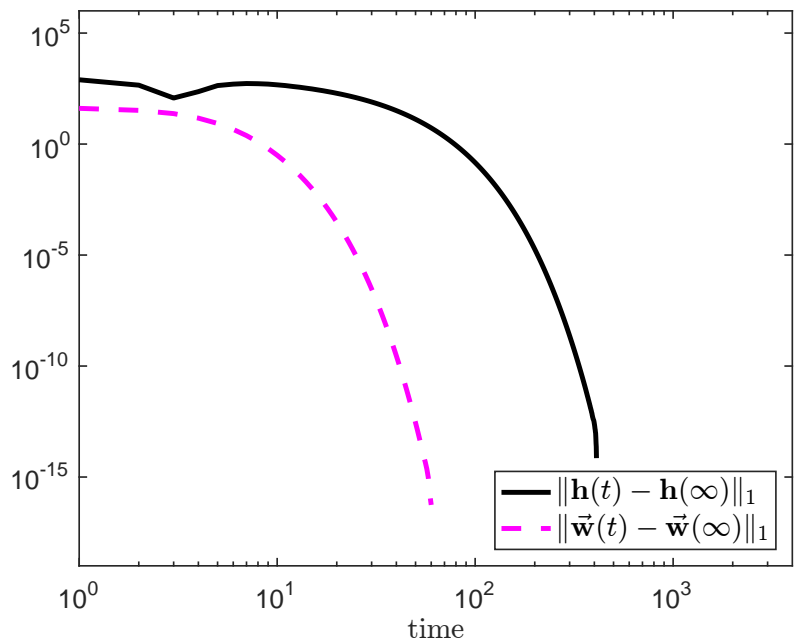}%
	}
	\caption{\label{fig:FE-dyn}
	The plot shows the MPA convergence on the graph $\GG_{FE}$, which contains a limited number of cycles. The solid black line is the error between the vector $\hh(t)$, which contains the node's harmonic influence values, and its limit $\hh(\infty)$; the dashed magenta line is the error between the vector $\vec\ww(t)$, which contains the messages $\Wij(t)$, and its corresponding limit $\vec\ww(\infty)$.}
\end{figure}

\begin{figure} 	\centering 	\fbox{%
	\includegraphics[trim={\figtriml} {\figtrimb} {\figtrimr} {\figtrimt}, 	clip, width={\figwidth}, keepaspectratio=true]{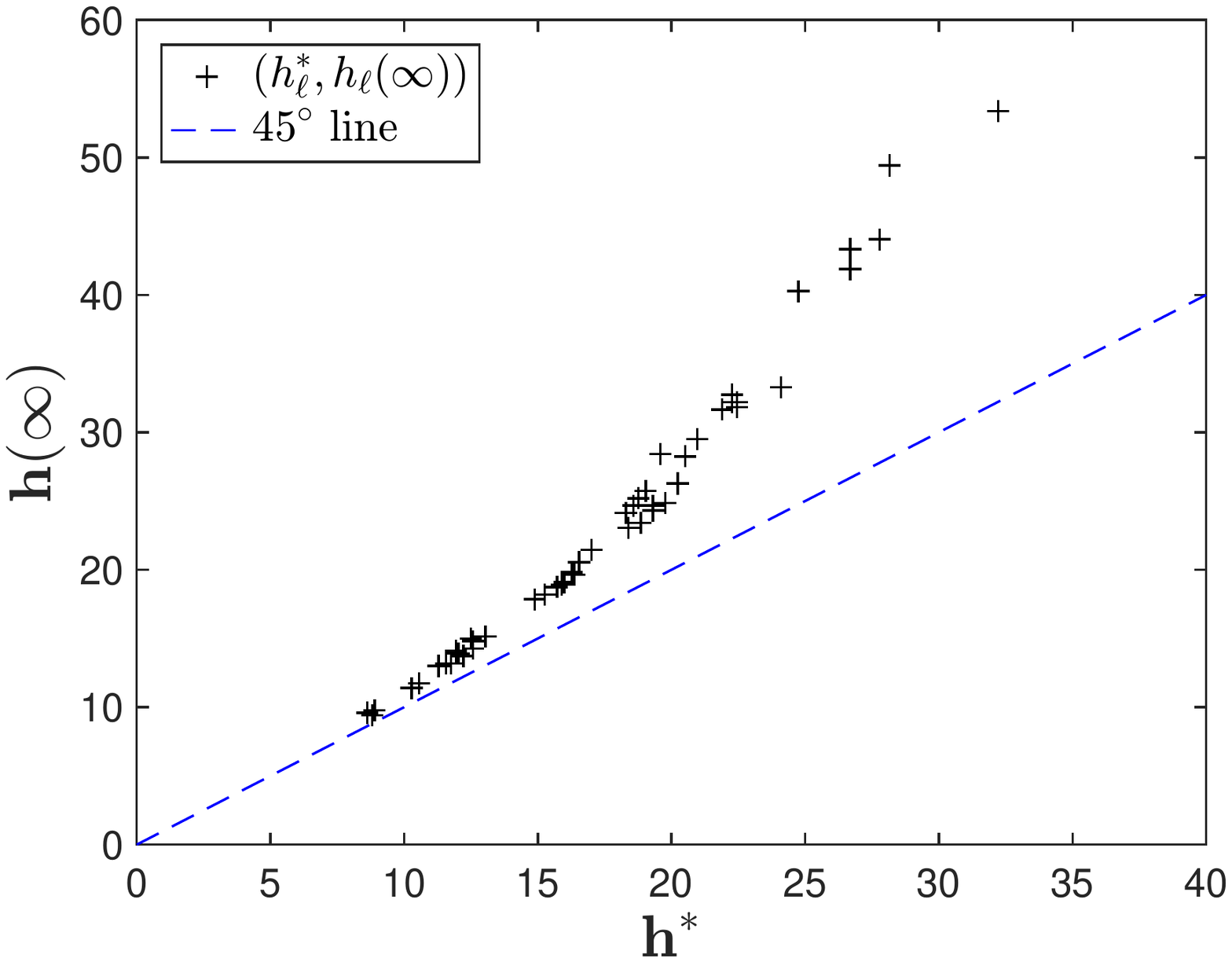}%
	}
	\caption{\label{fig:FE-HH} The elements of the vector $\hh(\infty)$ against the corresponding exact values of the harmonic influence, collected in the vector $\hh^*$, for the graph $\GG_{FE}$. 
All the points are above the $45^\circ$ line. }
\end{figure}

\begin{figure} 	\centering 	\fbox{%
	\includegraphics[trim={\figtriml} {\figtrimb} {\figtrimr} {\figtrimt}, 	clip, width={\figwidth}, keepaspectratio=true]{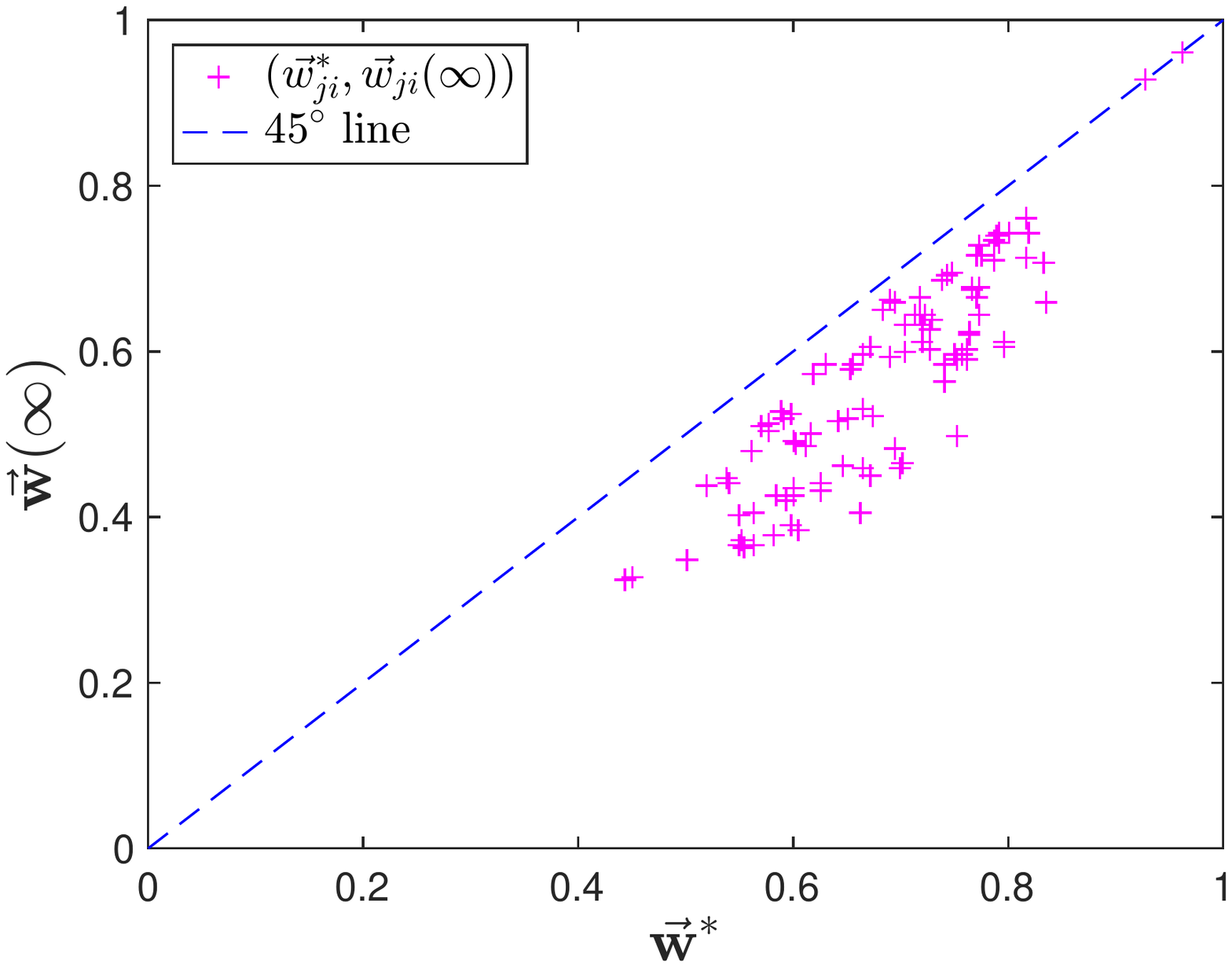}%
	}
	\caption{\label{fig:FE-ww}The elements of the vector $\vec\ww(\infty)$ against the corresponding exact values of the electric interpretation, collected in the vector $\vec\ww^*$, for the graph $\GG_{ST}$.  
	All the points are below or on the $45^\circ$ line. }
\end{figure}

The next simulation regards the graph $\GG_{FE} = (I, E_{FE})$, represented in Figure~\ref{fig:FE-graph}. This graph has 59 edges for 50 nodes: it contains 10  edges more than $\GG_{ST}$, which form a few cycles, and has diameter 9. 
Figure~\ref{fig:FE-dyn} shows the convergence time of the MPA. The distance between $\vec\ww(t)$ and $\vec\ww(\infty)$ (dashed magenta line) becomes negligible after 60 iterations.
The distance between $\hh(t)$ and $\hh(\infty)$ (solid black line) requires about 400 iterations to become negligible. 

The MPA is not exact on the graph $\GG_{FE}$, but the nodes' rankings  implied by the harmonic influence are nearly preserved.
Figure~\ref{fig:FE-HH} represents the limit values of the harmonic influence (the elements of the vector $\hh(\infty)$) against the corresponding exact values (the elements in the vector $\hh^*$). 
All crosses are above the $45^\circ$ line, a behaviours consistently observed throughout simulations. 
If the crosses align to form a strictly monotonically increasing function, the nodes' rankings implied by the harmonic influence are completely preserved. To give a quantitative evaluation of how much the rankings are preserved we use the Spearman's correlation coefficient, that for the two vectors $\hh^*$ and $\hh(\infty)$ of this simulation is 0.9940.
The magenta crosses in Figure~\ref{fig:FE-ww} serve to compare the elements of the vector $\vec\ww(\infty)$ against the corresponding exact values of the electric interpretation, in the vector $\vec\ww^*$.  
All the points are below or on the $45^\circ$ line, and form an elongated cloud.

The last simulation is about the Erd\H{o}s-R\'enyi random graph $\GG_{ER} = (I, E_{ER})$, represented in Figure~\ref{fig:ER-graph}. 
The graph $\GG_{ER} $ contains 123 edges, that form many cycles, and its  diameter is 5.
Figure~\ref{fig:ER-dyn} shows the convergence time of the MPA. The distance between $\vec\ww(t)$ and $\vec\ww(\infty)$ (dashed magenta line) becomes negligible after 30 iterations, about a half of the number of iterations required by the graph $\GG_{FE}$.
Instead, the distance between $\hh(t)$ and $\hh(\infty)$ (solid black line) requires about 2500 iterations to become negligible. 

Also on the graph $\GG_{ER}$ the MPA is not exact, but the nodes' rankings are nearly preserved.
Figure~\ref{fig:ER-HH} represents the elements of the vector $\hh(\infty)$ against the corresponding elements in the vector $\hh^*$. All crosses are above the $45^\circ$ line and nearly aligned in a sort of parabola. The largest value of $h_\ell(\infty)$ is about 5 times bigger than the corresponding $h^*_\ell$. 
The Spearman's coefficient for the two vectors $\hh^*$ and $\hh(\infty)$ of this simulation is 0.9939.
The magenta crosses in Figure~\ref{fig:ER-ww} compare the elements of $\vec\ww(\infty)$ against the corresponding exact values in $\vec\ww^*$. 
All the points are below or on the $45^\circ$ line, further than it was for $\GG_{FE}$.

\begin{figure} \centering \fbox{%
	\includegraphics[ 
	trim={35mm} {100mm} {35mm} {78mm}, 
	clip, width={\figwidth}, keepaspectratio=true]{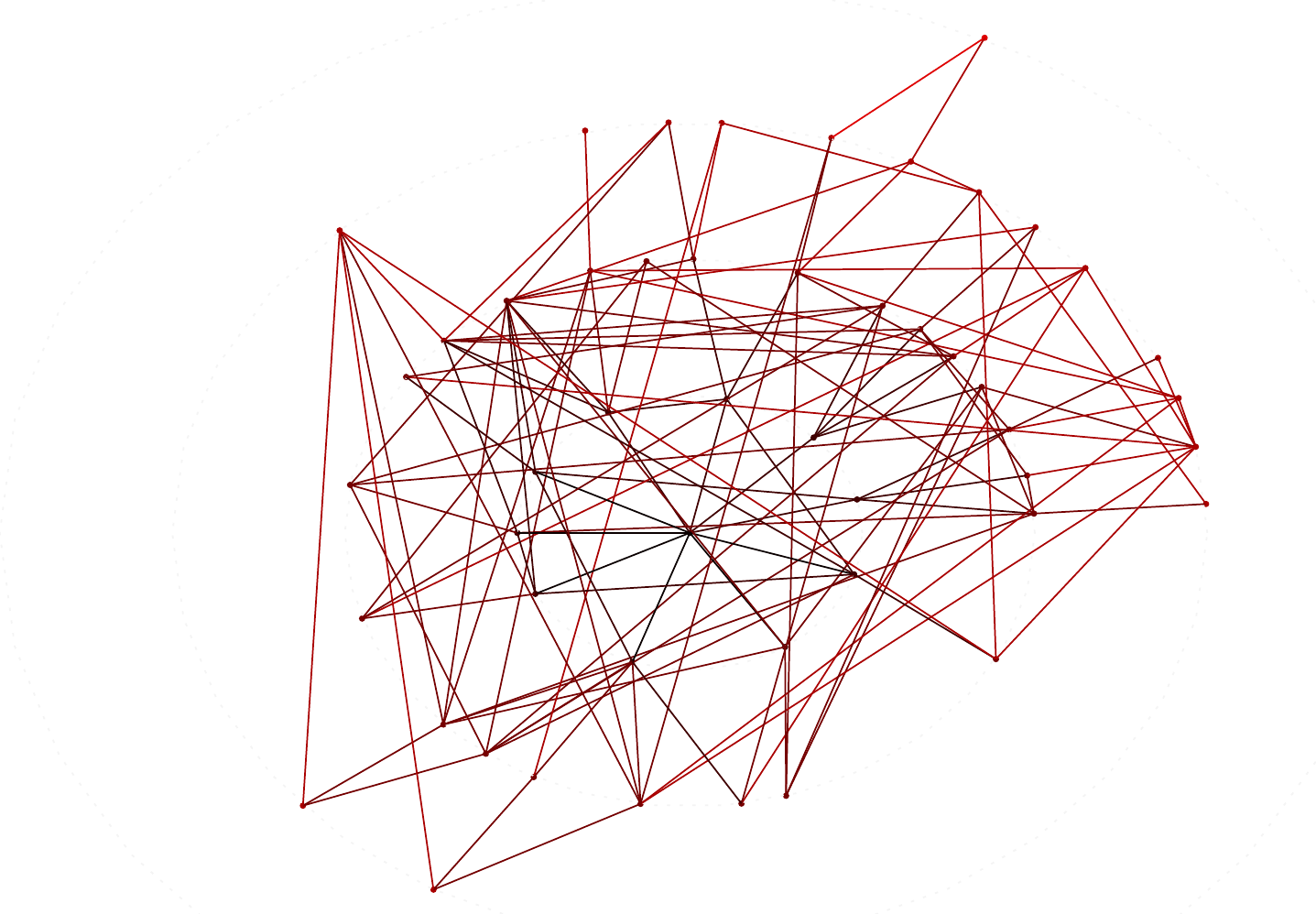}%
	}
	\caption{\label{fig:ER-graph}Radial plot of the Erd\H{o}s-R\'enyi random graph $\GG_{ER} = (I, E_{ER})$ that contains $50$, $123$ edges, and has diameter 5.}
\end{figure}

\begin{figure}	\centering 	\fbox{%
	\includegraphics[trim={\figtriml} {\figtrimb} {\figtrimr} {\figtrimt}, 	clip, width={\figwidth}, keepaspectratio=true]{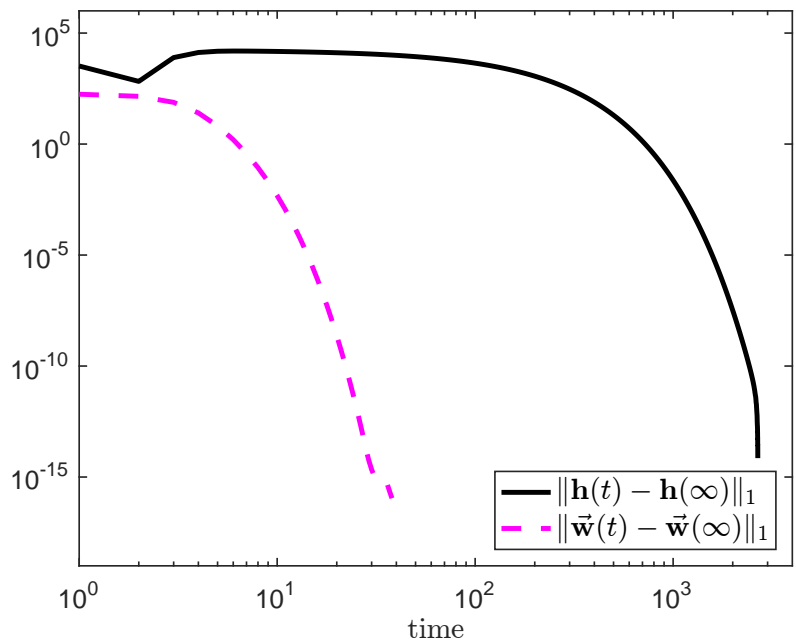}%
	}
	\caption{\label{fig:ER-dyn}The MPA convergence on the Erd\H{o}s-R\'enyi random graph $\GG_{ER} = (I, E_{ER})$.  The solid black line is the error between the vector $\hh(t)$, which contains the node's harmonic influence values, and its limit $\hh(\infty)$; the dashed magenta line is the error between the vector $\vec\ww(t)$, which contains the messages $\Wij(t)$, and its corresponding limit $\vec\ww(\infty)$.}
\end{figure}

\begin{figure} 	\centering 	\fbox{%
	\includegraphics[trim={\figtriml} {\figtrimb} {\figtrimr} {\figtrimt}, 	clip, width={\figwidth}, keepaspectratio=true]{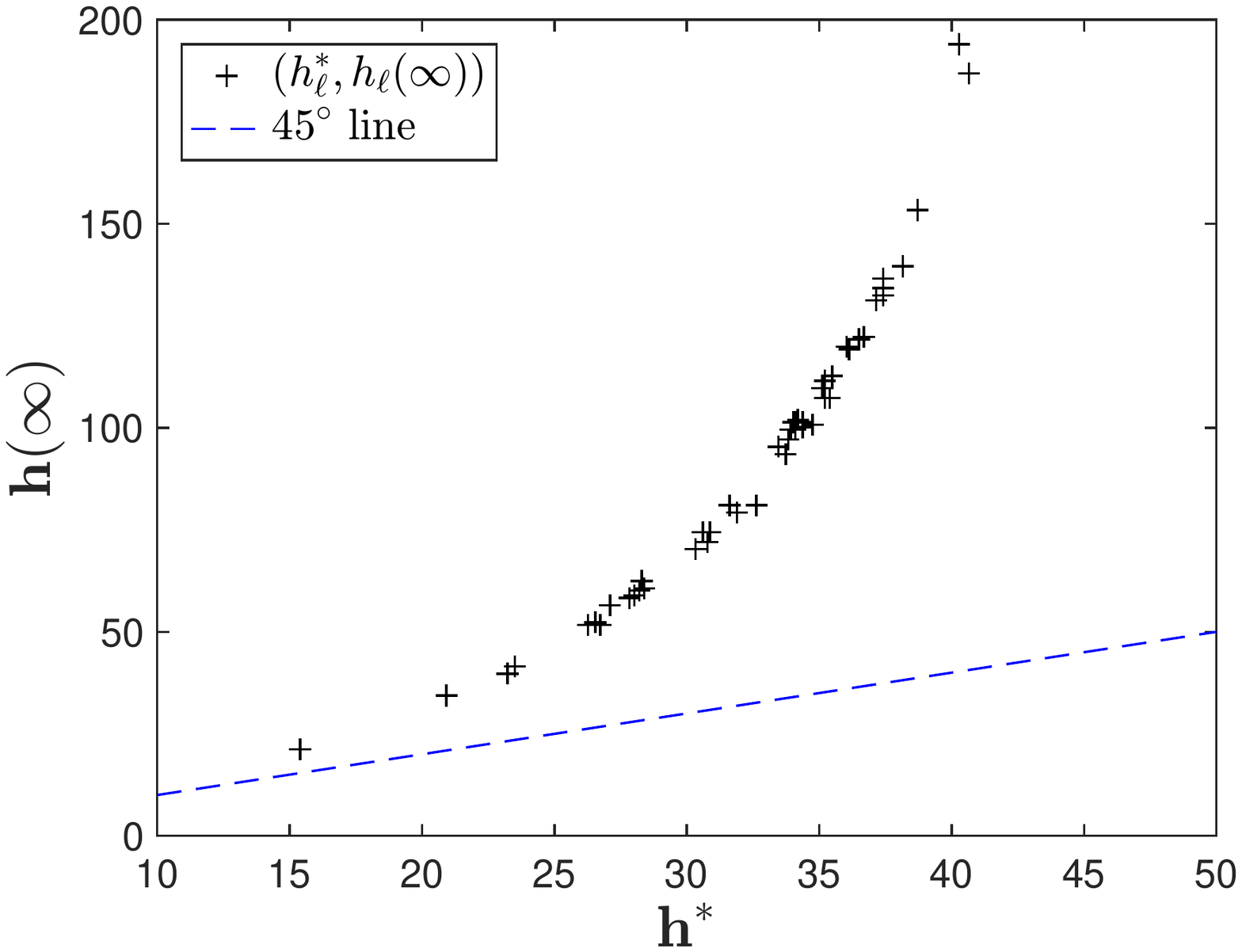}%
	}
	\caption{\label{fig:ER-HH}The elements of the vector $\hh(\infty)$ against the corresponding exact values of the harmonic influence, collected in the vector $\hh^*$, for the graph $\GG_{ER}$.  All the points are above the $45^\circ$ line.  }
\end{figure}

\begin{figure} 	\centering 	\fbox{%
	\includegraphics[trim={\figtriml} {\figtrimb} {\figtrimr} {\figtrimt}, 	clip, width={\figwidth}, keepaspectratio=true]{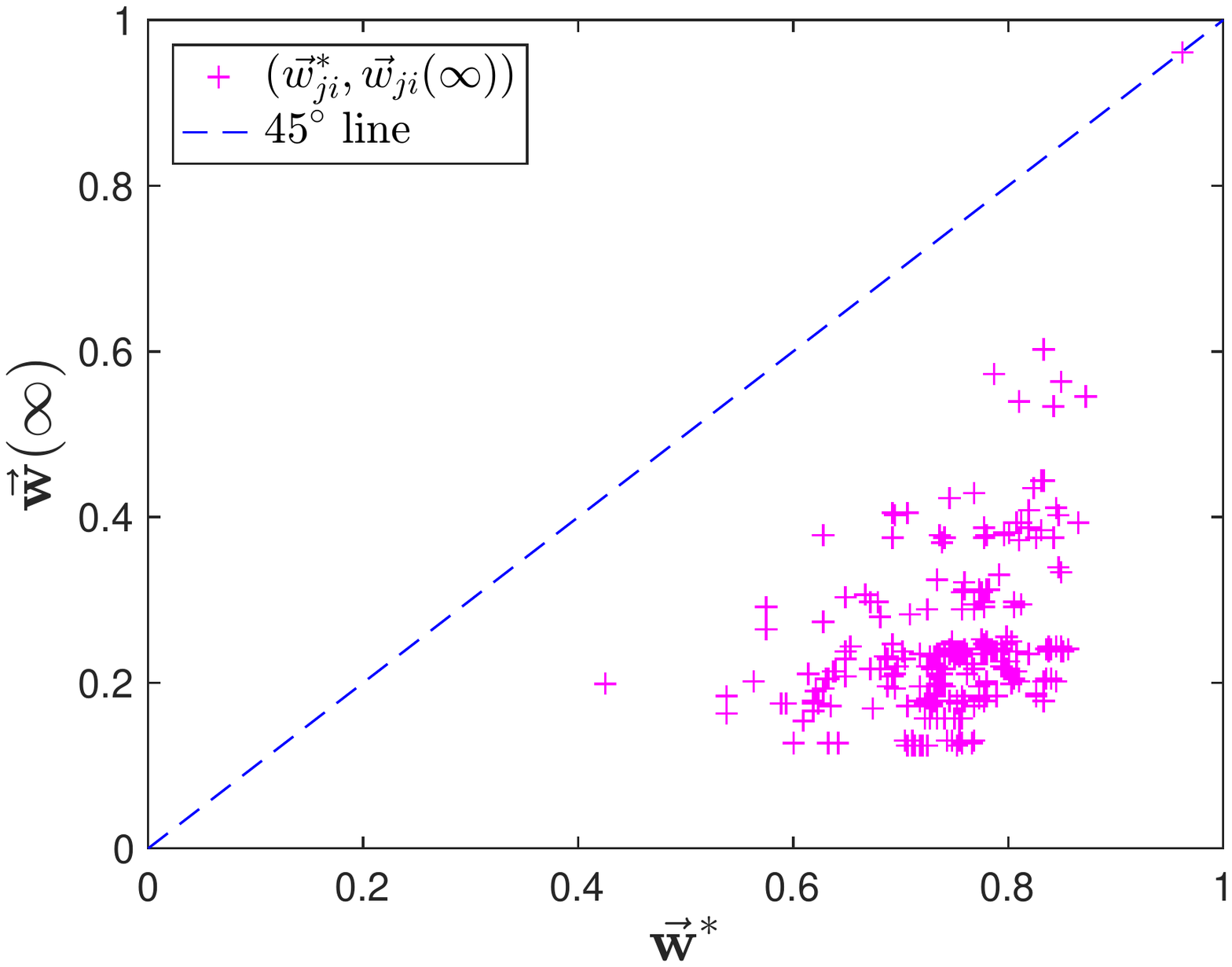}%
	}
	\caption{\label{fig:ER-ww}The elements of the vector $\vec\ww(\infty)$ against the corresponding exact values of the electric interpretation, collected in the vector $\vec\ww^*$, for the graph $\GG_{ER}$. 
	All the points are below or on the $45^\circ$ line. }
\end{figure}

On graphs containing loops, the asymptotic values of the messages $\Wij(t)$, i.e. the limits $\Wij(\infty)$, are smaller (or equal) than the corresponding exact values, computed using the electrical interpretation. 
The limit $\Wij(\infty)$ is exact if the graph is a tree, else represents the value that would be computed by $i$ on its \textit{computation tree} (i.e. the infinite ``unwrapped'' graph obtained exploring the neighborhood of $i$ in a breadth first manner~\cite{Vassio:2014:journal}). From an electrical point of view, this tree contains more paths to the reference nodes than the original graph and this fact makes the MPA compute a smaller resistance from node $i$ to the reference. 
At the same time, in graphs with loops, the MPA approximation of the harmonic influence of the nodes is consistently higher than the exact value. This is again due to the properties of the computation tree, which has more nodes than the original graph. This effect overcomes the fact that the limit messages  $\Wij(\infty)$ are smaller.

\section{Conclusion}\label{sec:conclu}
By extending some recent work in distributed influence maximization~\cite{Vassio:2014:journal}, 
in this paper we studied the harmonic influence of nodes in a opinion dynamic model 
in presence of a constant opinion field and a distributed message passing algorithm to compute it. 
As our main contribution, we proved the convergence of this algorithm on any graph, 
provided the matrix $Q$ satisfies a reversibility condition. Actually, we have observed in simulation that this assumption is not necessary for the algorithm to convergence and produce useful results: thus, future work could focus on finding milder convergence conditions.



\section*{Acknowledgment}
The authors wish to thank H. Zwart (University of Twente) and P. Van Dooren (Catholic University of Louvain) for fruitful discussions about this work.


\end{document}